\documentclass[12pt]{article}

 \usepackage{amsmath}
 \usepackage{amssymb}
 \usepackage{amsfonts}
 \usepackage{graphicx}
 \usepackage{amsthm}

\newcommand{\N}{\mathbb{N}}
\newcommand{\Z}{\mathbb{Z}}
\newcommand{\cX}{\mathcal{X}}
\newcommand{\cW}{\mathcal{W}}
\newcommand{\tcW}{\widetilde{\mathcal{W}}}
\newcommand{\tW}{\tilde{W}}
\newcommand{\bH}{\mathbb{H}}
\renewcommand{\Pr}{\mathbb{P}}
\newcommand{\X}{\mathcal{X}}
\newcommand{\cF}{\mathcal{F}}
\newcommand{\E}{\mathbb E \,}
\renewcommand{\emptyset}{\varnothing}

\newcommand{\diam}{{\rm diam}}
\newcommand{\F}{{\cal F}}
\newcommand{\tN}{{\tilde{N}}}
\newcommand{\tnu}{{\tilde{\nu}}}
\newcommand{\tR}{{\tilde{R}}}
\newcommand{\eps}{\varepsilon}
\def\bea{\begin{eqnarray}}
\def\eea{\end{eqnarray}}
\renewcommand{\epsilon}{\varepsilon}
\newcommand{\R}{\mathbb{R}}

\newcommand{\dist}{{\rm dist}}
\renewcommand{\log}{{\ln \:}}

\newtheorem{theo}{Theorem}

\newtheorem{prop}{Proposition}
\newtheorem{lemm}{Lemma}
\newtheorem{defn}{Definition}

\newcommand{\eqco}{\setcounter{equation}{0}}
\newcommand{\thco}{\setcounter{theo}{0}}
\newcommand{\prco}{\setcounter{prop}{0}}
\newcommand{\laco}{\setcounter{lemm}{0}}
\newcommand{\coco}{\setcounter{coro}{0}}
\newcommand{\cjco}{\setcounter{conj}{0}}

\newcommand{\deco}{\setcounter{defn}{0}}
\newcommand{\allco}{\eqco  \thco \prco \laco \coco \cjco \deco}

\begin{document}


\author{
Mathew D. Penrose
         \thanks{Department of Mathematical Sciences, University of Bath, Bath, BA2 7AY United Kingdom: {\texttt m.d.penrose@bath.ac.uk} }
        \thanks{Funded by EPSRC grant EP/T028653/1}
\\
{\normalsize{\em University of Bath}} }

\title{Record times for  coverage thresholds and maximal spacings }



\maketitle

\begin{abstract}
Let $X_1,X_2, \ldots $ be independent uniform random points in
a bounded region $A \subset \R^d$ having a smooth boundary, $d \geq 1$.
	Let $B \subset A$ be compact.
	The {\em coverage
threshold} of $B$,  $R_n$, is the smallest $r$ such that $B$ is covered
by the balls of radius $r$ centred on $X_1,\ldots,X_n$. 
	The {\em maximal spacing} $\tR_n$ is the volume of the largest
	ball contained in $A \setminus \{X_1,\ldots,X_n\}$.

	We investigate the asymptotic frequency of {\em record times}
	in the  sequence $(R_n)$, 
	that is times $n$ for which $R_n < R_{n-1}$. 
	Let $N_m$ denote the number of records in the
	sequence $(R_n)$ up to time $m$, and let $\nu_m$ 
	be the time at which the $m$th record value
	of the sequence $(R_n)$ occurs.
	For $B \subset A^o$, we show that almost surely,
	$N_n \sim \frac12 (\log n)^2$ and
	$\nu_n^{1/\sqrt{n}}\to \exp \big(\sqrt{2}\: \big)$ as $n \to \infty$,
	and likewise for $\tN_n$ and $\tnu_n$, defined analogously in
	terms of $(\tR_n)$. But if $B=A$ and $d \geq 3$, then 
	$N_n \sim \frac12 (1- \frac{1}{d}) (\log n)^2$ and
	$\nu_n^{1/\sqrt{n}}\to \exp \big( \sqrt{2d/(d-1)} \: \big)$.

	We also discuss the generalization (for fixed $k \in \N$)
	to $k$-coverage thresholds,
	maximal $k$-spacings and non-uniformly distributed points
	$X_i$ in $A$.
\end{abstract}

%


\section{Introduction}
\label{SecIntro}

The {\em random coverage} problem is the
question of whether a specified compact region $B$ in a Euclidean
space is covered by a union of balls (of equal radius)
centred on  a set of points
placed randomly in a region $A \supset B$ (possibly $A=B$).
This type of question has been much studied
\cite{BB,CSKM,Flatto,HallZW,HallBk,Janson,MPCov}. Potential applications include
wireless communications \cite{BB,Lan}, ballistics \cite{HallBk}, genomics
\cite{Roy} and topological data analysis \cite{BW,KTV}.

As more points get added, we can cover $B$
using smaller balls. We consider here the {\em coverage threshold} $R_n$,
defined (as in \cite{MPCov}) as the  the smallest radius of balls, centred on a
 sample $\{X_1,\ldots,X_n\}$ of $n$ random points in $A$, required 
to cover $B$. More generally, for $k \in \N$
 the {\em $k$-coverage threshold} $R_{n,k}$
is the smallest radius required to cover $B$ $k$ times. 
These thresholds are random variables, because the locations of
the centres are random.
We investigate their probabilistic behaviour 
as $n$ becomes large.

A related concept is the {\em maximal spacing} in the sample.
If $d=1$ this is the largest gap between successive points of the
ordered sample. For general $d$ and $k$, we define the maximal
$k$-spacing $\tR_{n,k}$ to be the volume of the largest `hole', defining
a hole to be a ball inside $A$ containing fewer than $k$ points of the sample
(note that the definition of $\tR_{n,k}$ does not involve $B$).
The maximal spacing has also been much studied; see for example
\cite{Aaron,Deheuvels,Henze,Janson2}. As described in \cite{Aaron},
there are many statistical applications.

Quite a lot is known about the asymptotic behaviour of $R_n$ and
$\tR_n$
including the asymptotic distribution
and strong laws of large numbers;
see  \cite{MPCov}.
In the present paper we investigate the behaviour of $R_n$ and of
$\tR_n$ as {\em time series} in $n$ (viewing $n$ as time).
In particular, we investigate
the limiting behaviour for the number of {\em records}
of $R_n$, that is, the number
of times $m$ up to time $n$ that we have $R_m < R_{m-1}$.
We also consider the number of records of $\tR_n$.
More generally, for fixed $k \in \mathbb N$ we consider records of $R_{n,k}$
and of $\tR_{n,k}$.
General motivation comes from the fact that modern datasets are often not only very large but rapidly increasing in size, so there may be something to be gained from understanding quantities such as $R_n$ and $\tR_n$ as
a time series in $n$. The theory  of records for the maximum or
minimum of an i.i.d. sequence is much more fully developed; see for example
the books \cite{ANbook,Arnold}.

For $m \in \N$, let $N_{m,k}$ (resp. $\tN_{m,k}$) denote the number of
records of the sequence $(R_{n,k})_{n \geq k}$
(resp. $(\tR_{n,k})_{n \geq k}$) up to time $m$
(since our sequence
$R_{n,k}$ is non-increasing in $n$, all of our records are {\em lower records}).
    Let $\nu_{m,k}$ (resp. $\tilde{\nu}_{m,k}$) be the time of the  
$m$th record value of the sequence $(R_{n,k})_{n \geq k}$ 
(resp.,  the sequence 
$(\tR_{n,k})_{n \geq k}$). 
In Theorem \ref{t:record} we show for $B \subset A^o$
that for any fixed $k \in \N$, 
\begin{align}
N_{n,k} \sim \frac12 (\log n)^2 ~~~ \mbox{ and } ~~~
	\nu_{n,k}^{1/\sqrt{n}} \to \exp \big( \sqrt{2} \:
	\big) ~~~ \mbox{  as } n \to \infty, ~~ a.s.,
	\label{e:takehome1}
\end{align}
but in Theorem \ref{Thm3} we show for $B=A$ with a nice boundary
that if $d \geq 3 $ then
\begin{align}
N_{n,k} \sim \frac{d-1}{2 d} (\log n)^2 ~~~ \mbox{ and } ~~~
	\nu_{n,k}^{1/\sqrt{n}} \to \exp \big(\sqrt{(2d)/(d-1)}\: \big) ~~~ \mbox{  as } n \to \infty, ~~ a.s.
\label{e:takehome2}
\end{align}
In Theorem \ref{Thm2} we show 
that $\tN_{n,k} \sim \frac12 (\log n)^2$ and
$\tnu_{n,k}^{1/\sqrt{n}} \to e^{2^{1/2}} $  as $ n \to \infty,$
a.s.

These contrast  with a classic result of R\'{e}nyi \cite{Renyi}
on record values in an i.i.d. $\R$-valued sequence, which says that
the number of records up to time $n$ 
in such a sequence is asymptotic to $\log n$.
Loosely speaking, in the i.i.d. situation the probability of a new observation
at time $n$ being a record is $1/n$, while in our situation it is about
$(\log n)/n$ when $B \subset A^o$,
and is about $(1-1/d) (\log n)/n$ when $B=A$;
since the partial sums of the series 
$\sum (\log n)/n$ are asymptotic to $(\log n)^2/2$ (see Lemma 
\ref{lemsum})  this suggests the stated asymptotics for the
number of records.

The reason for the $(\log n)/n$ asymptotics is roughly
as follows. Consider first the 
case of the maximal spacing in $d=1$, taking $A=[0,1]$. It is known
\cite{Slud} that the
largest spacing is almost surely asymptotic to $(\log n)/n$,
and a new observation causes the maximal spacing to decrease precisely when it
lands in the largest gap. 
In higher dimensions, again the probability content of the `largest hole'
is asymptotic to $(\log n)/n$ \cite{Deheuvels,MPCov}
so the same asymptotics apply; again, a record occurs when a new point arrives in the largest existing hole. For the
coverage threshold when $B \subset A^o$ similar heuristics apply, leading to
 \eqref{e:takehome1}.

 For the coverage threshold
 in the case $B=A$, we allow `holes' that are centred 
 in $A$ rather than having to be contained in $A$.  
 At least for $d \geq 3$, the largest hole is likely to be centred
 near the boundary $\partial A$ of $A$,
 and this leads to its probability content being
 asymptotic to $(1-1/d) (\log n)/n$ rather than $(\log n)/n$, and hence to the different constants arising in \eqref{e:takehome2}.

 One step  in converting the above heuristics into a proof is to show
 that the `largest hole' is almost surely {\em unique}, i.e. there
 is no `tie for first place' when ordering the holes by size.
 We need this uniqueness to be sure that the event of a new random point landing in the existing largest hole really
 does cause the $k$-coverage threshold or maximal $k$-spacing to decrease. 
 The uniqueness of the largest hole is   
 intuitively unsurprising but seems to require some work to prove rigorously. 
 See Section \ref{s:uniqueness}.

 The discussion so far  has been for uniformly distributed points in $A$,
 but in fact we shall state our results in a more general setting, allowing
 for a class of non-uniform distributions on $A$ for the random points $X_i$.
 However, we do not present any result for the asymptotic frequency 
 of records for the coverage threshold when $B=A$ and  
 the distribution of the random points is uniform over $A$ with $d=2$.
 This particular case is harder to deal with because the largest hole
 could be either in the interior of $A$ or on the boundary
 (see \cite{MPCov}), making
 the asymptotics more delicate. We leave this case for future work.



\section{Statement of results}
\allco
We work within the following mathematical framework.  Let $d \in \N$,
and let $A \subset \R^d$ be compact with non-empty interior.
Suppose $X_1,X_2,\ldots$ are independent identically distributed
random $d$-vectors with probability density function
$f$ with support
 $A $. We refer to the special case where $f$ is constant on $A$
 as the {\em uniform case}.
 Let $B \subset A$ be a  specified compact set (possibly
the set $A$ itself).
For $x \in \R^d$ and $r>0$, set $B(x,r):= \{y \in \R^d:\|y-x\| \leq r\}$,
where $\|\cdot\|$ denotes the Euclidean norm.
For $n \in \N$,  let $\X_n:= \{X_1,\ldots,X_n\}$.
Given also $k \in \N$, we
 define the $k$-coverage threshold $R_{n,k}$ of $B$
 (also written $R_{n,k}(B)$ if the choice of $B$ needs to be made clear)
 by
\begin{align}
R_{n,k} : =
 \inf \left\{ r >0: \X_n   (B(x,r)) \geq k 
~~~~ \forall x \in B \right\},
~~~ n,k  \in \N,
	\label{e:defRnk}
\end{align}
where for any point set $\X \subset \R^d$ and any $D \subset \R^d$ we write
$\X(D)$ for the number of points of $\X$ in $D$,
 and we use the convention $\inf\{\} := +\infty$.
In particular $R_n : = R_{n,1}$ is the coverage threshold.  Observe that
$R_n  = \inf \{ r >0: B \subset \cup_{i=1}^n B(X_i,r) \}$.
If $B=A$ then $R_n$ is the Hausdorff distance between the sets $A$ and
$\X_n$.

Let $\theta$ denote the volume of the unit ball in $\R^d$,
that is $\theta = \pi^{d/2}/\Gamma((d/2)+1)$.  Set
\begin{align}
\tilde{R}_{n,k} : = \inf \left\{ r >0: \X_n (B(x,r))
\geq k ~~~ \forall x \in  A^{(r)} \right\},~~~ n,k \in \N, ~~ n \geq k,
\label{eqmaxspac}
	\end{align}
where $A^{(r)} := \{x \in A: B(x,r) \subset A\}$, the `$r$-interior' of $A$.  
Define the {\em maximal $k$-spacing} to be $\theta \tilde{R}_{n,k}^d$.
Set $\tR_n: = \tR_{n,1}$; the maximal 1-spacing $\theta \tilde{R}^d_n$
is also called the maximal spacing.  
Note that $\tilde{R}_{n,k} $ is the largest $r$ 
such that we can find an open ball of radius $r$ in $A$ containing
at most $k-1$ points of $\cX_n$.

For $n > k$, we define the number of (lower)
records for the $k$-coverage threshold, $N_{n,k}$, by
\begin{align*}
N_{n,k} :=  1+
	\sum_{i = k+1}^n {\bf 1}_{\{R_{i,k} < \min_{j \in \{k,\ldots,i-1\}}
	R_{j,k}\}}
= 
1 + \sum_{i = k+1}^n {\bf 1}_{\{R_{i,k} < R_{i-1,k}\}},
\end{align*}
where the second equality holds because $(R_{n,k})_{n \geq k}$ is
 monotone nonincreasing in $n$.
Also let $\nu_{m,k}$ denote the time of the $m$th record; that is
$\nu_{1,k} := k$ and for each $m >1$ in succession set
 $\nu_{m+1,k} := \inf\{n > \nu_{m,k}: R_{n,k} < R_{n-1,k}\}$.
 Similarly for the maximal $k$-spacing, we define the number of records 
$$
\tN_{n,k} := 1 + \sum_{i = k+1}^n {\bf 1}_{\{\tR_{i,k} < \tR_{i-1,k}\}},
$$
with
$\tnu_{1,k} :=k $ and for each $m >1$
in succession,
 $\tnu_{m+1,k} := \inf\{n > \tnu_{m,k}: \tR_{n,k} < \tR_{n-1,k}\}$.

For $D \subset \R^d$
let $D^o$ denote the interior and $\overline{D}$ the closure of $D$,
and let $\partial D: = \overline{D} \setminus D^o$.
Also let $\dim(D)$ denote the Hausdorff dimension of $D$.
We assume throughout that 
$A$ and $B$ are compact with
$\overline{A^o} = A$ and
  either 
 $B \subset A^o$, or $B=A$.
 We also assume throughout that the density function $f$ is continuous on $A$
 with $f_0 >0$, where 
 we define the constants
 \begin{align}
 f_0:= f_0(B) := \inf_{x \in B} f(x): ~~~~~~~~~
 f_1:=  \inf_{x \in \partial A} f(x).
	 \label{e:f0f1def}
 \end{align}

 For $d \geq 2$,
we say that $A$ {\em has a $C^{1,1}$ boundary} (for short: $\partial A \in
C^{1,1}$) if
for each $x \in \partial A$
there exists a neighbourhood $U$ of $x$ and a real-valued function $\phi$ that
is  defined on an open set in $\R^{d-1}$
and Lipschitz-continuously differentiable, such
that  $\partial A \cap U$, after a rotation, is the graph of the
function $\phi$.

For $d=1$, the condition of $A$ having a $C^{1,1}$ boundary amounts
to saying that $\partial A$ is a finite set, or in other words $A$
is a finite union of intervals, and this is how we interpret the
$C^{1,1}$ boundary condition when $d=1$. With this interpretation of
the $C^{1,1}$ condition, all of our results hold for all $d$ including $d=1$; some parts of the proofs have simplifications for the case $d=1$ that we shall not discuss in detail.

	Given $g: \N \to \R$ and $h: \N \to [0,\infty)$,
	we write $g(n) = O(h(n))$ as $n \to \infty$ if
	$\limsup_{n \to \infty} (|g(n)|/h(n)) < \infty$.
	We write $g(n) \sim h(n)$ if $\lim_{n \to \infty} g(n)/h(n)=1$.
	Similarly, given $\delta >0$ and $g: (0,\delta) \to \R$
	and $h: (0,\delta) \to [0,\infty)$, 
	we write $g(x) = O(h(x))$ as $x \downarrow 0$ if
	$\limsup_{x \downarrow 0 } (|g(x)|/h(x)) < \infty$.


\begin{theo}[Records for the $k$-coverage threshold when $B \subset A^o$]
\label{t:record}
	Suppose that $B = \overline{B^o} \neq \emptyset $,
	and that  $B \subset  A^o$ with $\dim(\partial B) < d $.
Let $k \in \N$.
  Then as $n \to \infty$,
	\begin{align}
		\E[N_{n,k}] & \sim (1/2) (\log n )^2;
\label{0125a}
\\
		N_{n,k} & \sim  (1/2) (\log n )^2
~~~~{\rm a.s.}
\label{0125b}
	\end{align} 
and 
	\begin{align}
	\nu_{n,k}^{1/\sqrt{n}} \to \exp \big( \sqrt{2} \: \big)   
~~~~{\rm a.s.}
\label{0125c}
		\end{align}
\end{theo}
Theorem \ref{t:record} also applies (by minor modifications of the proof)
if $B=A$ is a torus.
\begin{theo}[Records for the maximal $k$-spacing]
	\label{Thm2}
	Suppose  $\partial A \in C^{1,1}$ or $A$ is convex.
	Let $k \in \N$.
	Then as $n \to \infty$,
	\eqref{0125a}, \eqref{0125b} and \eqref{0125c} all hold
	with
	$N_{n,k}$ replaced by $\tN_{n,k}$
	and $\nu_{n,k}$ replaced by $\tnu_{n,k}$.
\end{theo}
\begin{theo}[Records for the $k$-coverage threshold when $B = A$]
	\label{Thm3}
	Suppose $B=A$ and $\partial A \in C^{1,1}$.
	Let $k  \in \N$.
	If $1/f_0 > (2-2/d)/f_1$ 
	then \eqref{0125a}, \eqref{0125b} 
	and \eqref{0125c} all  hold as $n \to \infty$. If instead
	 $1/f_0 < (2-2/d)/f_1$, then  as $n \to \infty$,
	\begin{align}
		\E[N_{n,k}] & \sim (1/2) (1-1/d)(\log n )^2;
\label{0720a}
\\
		N_{n,k} & \sim (1/2) (1-1/d) (\log n )^2
~~~~{\rm a.s.}
\label{0720b}
	\end{align}
and 
	\begin{align}
	\nu_{n,k}^{1/\sqrt{n}} \to \exp \big( \sqrt{2 d/(d-1)}  \: \big)   
~~~~{\rm a.s.}
\label{0720c}
	\end{align}
\end{theo}
In the uniform case with $B=A$, $f_0= f_1$ so by Theorem \ref{Thm3},
\eqref{0125a}, \eqref{0125b} 
	and \eqref{0125c} hold if $d =1$
	 but
	\eqref{0720a}, \eqref{0720b} and \eqref{0720c}
	  hold if $d \geq 3$. When
	$d=2$ (in the uniform case with $B=A$) we do not present any results,
	but based on \cite[Theorem 3.2]{MPCov} we
	conjecture that if $k \geq 2$ then
	\eqref{0720a}, \eqref{0720b} and \eqref{0720c}
	hold, while if $k=1$ then the limiting value 
	of $2 N_{n,k}/(\log n)^2$ will be a weighted average of
	$1$ and $1/2$, with the weights depending on
	the perimeter of $A$.

	More generally, in cases with $B=A$, $\partial A \in C^{1,1}$
	and $(2-2/d)/f_1 = 1/f_0$, we do not present any results.
	Another extension that we do not consider here is to the intermediate
	case where $B \neq A$ but $B$ is not contained in $A^o$.

	Two other related quantities  that are amenable to our
	methods are the number of records for the $k$-coverage 
	threshold when $B=A$ is a convex polytope, and the number of
	`pseudo-records' for the largest $k$-nearest neighbour link. We
	discuss these quantities in Section \ref{s:related}
	at the end of the paper.

	\section{A general lemma}
	We shall prove Theorems \ref{t:record}, \ref{Thm2} and \ref{Thm3}
	by applying the following general probabilistic result
	in each of our three situations.

\begin{lemm}
	\label{l:genrec}
	Let $k \in \N$.
	Suppose $(\cF_n)_{n \geq k}$ is a filtration and
	$(I_n,U_n,Y_n)_{n \geq k}$
	is an adapted sequence of random 3-vectors with $I_n$
	and $U_n$ Bernoulli distributed for all $n$, 
	such that for all $n \geq k+1$ we have 
	\begin{align}
		\E [ I_nU_{n-1} | \F_{n-1} ] = Y_{n-1} U_{n-1} ~~~~ {\rm a.s.}
\label{0116c}
	\end{align}
	and 
	\begin{align}
		\E[ I_n | \F_{n-1} ] \leq Y_{n-1} ~~~~ {\rm a.s.}
	\label{e:squeeze}
	\end{align}
	Suppose moreover that 
	\begin{align}
		\sum_{n=k}^\infty (1- U_n)  < \infty, ~~~~~{\rm a.s.},
		\label{e:Enev}
		\end{align}
	and that,  almost surely, $Y_n \sim (c\, \log n/n)$ as
	$n \to \infty$ for some constant $c >0$,
	and that $\{n Y_n/\log n, n \geq k\}$ is a uniformly integrable
	family of random variables. For $n \in \N \cap [k,\infty)$ set
	$N_n := \sum_{i=k}^n I_i$, set
$\nu_{1} =k$ and for each successive $m \in \N$ set
	$\nu_{m+1} = \inf\{n > \nu_{m}: I_n =1 \}$. 
	Then $\nu_n$ is a.s. finite for all $n$, and as $n \to \infty$ we have 
	\begin{align}
		\E[N_{n}] & \sim (c/2)( \log n )^2;
\label{0125ac}
		\\
		N_{n} & \sim (c/2)( \log n )^2
~~~~{\rm a.s.}
\label{0125bc}
	\end{align} 
and
	\begin{align}
	\nu_{n}^{1/\sqrt{n}} \to \exp \big( \sqrt{2/c} \: \big)   
~~~~{\rm a.s.}
\label{0125cc}
	\end{align}
\end{lemm}
The case $U_n \equiv 1$ of the
lemma says that if $Y_n$ are known to satisfy $(\log n)/n$ asymptotics,
and $Y_n$ is the conditional expectation of the indicator variable 
$I_{n+1}$
then the sum $\sum_{i \leq n} I_i$ has $(\log n)^2$ asymptotics. 
The  general statement allows for a relaxation whereby the
conditional expectation identity holds only on event $\{U_n=1\}$.

In the application of Lemma \ref{l:genrec} to the proof of our theorems,
$I_n$ will be the indicator of a new record arising at time $n$ and
$Y_n$ will be the size of the largest `hole' at that time, while $U_n$ will
be the indicator of the event that there is a unique largest `hole' at time $n$.

The following analytical lemma is required for the proof of Lemma
	\ref{l:genrec}.

\begin{lemm}
\label{lemsum}
Suppose $(x_n)_{n \geq 1}$ is a sequence of real numbers
with $x_n \sim (\log n)/n$ as $n \to \infty$. Then 
for any $\ell  \in \N$,
$$
\sum_{i=\ell +1}^n x_i \sim \frac{(\log n)^2}{2}  ~~~~ {\rm as} ~~ n \to \infty.
$$
\end{lemm}
\begin{proof}
Let $\eps \in (0,1)$. Since $(\log x)/x$ is decreasing in $x$ for $x \geq e$,
there exists $\ell_0 \in \N$ such
that for all $\ell \geq \ell_0$  and $n > \ell$,
	\begin{align*}
	\sum_{i=\ell+1}^n x_i  & \leq (1+ \eps) \sum_{i=\ell+1}^n \frac{\log i}{i} 
\leq (1+ \eps) \int_{\ell}^n \frac{\log x}{x} dx 
\\
	& = (1+ \eps) \left( \frac{(\log n)^2}{2} -
\frac{(\log \ell)^2}{2} \right),
	\end{align*}
and hence  for any (fixed) $\ell \geq \ell_0$, but then also for 
any $\ell \geq 1$, we have
$$ 
\limsup_{n  \to \infty} \frac{2}{(\log n)^2} \sum_{i= \ell+1}^n x_i 
\leq (1+ \eps).
$$
Also, there exists $\ell_1 \in \N$ such that for all $ n > \ell \geq \ell_1$ 
we have
	\begin{align*}
		\sum_{i=\ell+1}^n x_i &  \geq (1- \eps) \sum_{i=\ell+1}^n \frac{\log i}{i} 
\geq (1- \eps) \int_{\ell+1}^{n+1} \frac{\log x}{x} dx 
\\
		&	=  (1- \eps) \left( \frac{(\log (n+1))^2}{2} - \frac{(\log (\ell+1))^2}{2} \right)
	\end{align*}
and hence  for any (fixed) $\ell \geq \ell_1$, but then also for 
any $\ell \geq 1$, we have
$$ 
\liminf_{n  \to \infty} \frac{2}{(\log n)^2} \sum_{i= \ell+1}^n x_i 
\geq (1- \eps),
$$
and the result follows.
\end{proof}

 For $a \in \R$ we let $\lfloor a \rfloor := \sup(\Z \cap (-\infty,a])$
 and $\lceil a \rceil
 := \inf(\Z \cap [a,+\infty))$.

\begin{proof}[Proof of Lemma \ref{l:genrec}]
	Set $I_n^* := I_n U_{n-1}$ for $n > k $
	and $Y^*_n := Y_n U_n$ for $n \geq k$.
	By \eqref{e:Enev},  $Y_n^* \sim c \, (\log n)/n$ as $n \to \infty$, a.s.
Hence  by Lemma \ref{lemsum},
 almost surely
	\begin{align}
\sum_{i=k}^n Y_i^* \sim \frac{ c \, (\log n)^2 }{2}.
\label{0116b}
	\end{align} 
By the uniform integrability condition the
 random variables
	$(nY_n^*/\log n, n \geq k)$
are  also uniformly integrable.
	Therefore we have as $n \to \infty$ that
		\begin{align}
	\E[Y_n] \sim \frac{c \, \log n}{n};
			~~~~ ~~~~ ~~~~
			 ~~~~
	\E[Y_n^*] \sim \frac{c \, \log n}{n}.
\label{0308b}
		\end{align}
%
 Then by  Lemma \ref{lemsum} we obtain
	as $n \to \infty$
 that  almost surely 
 \begin{align}
	 \sum_{i = k+1}^n \E[Y^*_{i-1}] \sim \frac{c \, (\log n)^2}{2} 
	 ~~~~ \mbox{and} ~~~
	 \sum_{i = k+1}^n \E[Y_{i-1}] \sim \frac{c \, (\log n)^2}{2}.
	 \label{e:asympE}
 \end{align}
	By
	\eqref{0116c}, $\E[I_n |\cF_{n-1} ] \geq \E[I^*_n|\cF_{n-1}]
	= Y^*_{n-1}$ almost surely, so
 by \eqref{e:squeeze},
 $$
 \sum_{i = k+1}^n \E [ Y_{i-1}^*] \leq \sum_{i = k+1}^n \E[I_i]
 \leq \sum_{i = k+1}^n \E [ Y_{i-1}], 
 $$
 and combining this with \eqref{e:asympE},
 we obtain (\ref{0125ac}).

Next we prove \eqref{0125bc}.
For $n > k$, 
let $Z_n:= I^*_n - Y^*_{n-1}$, and 
	set
	$N_n^* := \sum_{i = k+1}^n I_i^*$. Then
			\begin{align}
N^*_{n} = \left( \sum_{i = k+1}^n Y^*_{i-1} \right) 
+ \sum_{i = k+1}^n Z_i. 
\label{0116d}
			\end{align}
By (\ref{0116c}) we have $\E[Z_n|\F_{n-1}]=0$ and by
	\eqref{e:squeeze} $Y_{n-1} \geq 0$ a.s.
Also by \eqref{0116c}, 
$$
\E[Z_n^2|\F_{n-1}] = (1-Y^*_{n-1})^2 Y^*_{n-1} + (Y^*_{n-1})^2 (1-Y^*_{n-1}) 
= Y^*_{n-1}(1-Y^*_{n-1}) \leq Y_{n-1}^*,
$$ 
so that
by (\ref{0308b}) we obtain as $n \to \infty$ that
$
\E[Z_n^2] = 
 O \big(n^{-1} \log n \big).
$
Setting $b_n : = (\log n)^2$, we have 
as $n \to \infty $ that
$
b_n^{-2} \E[Z_n^2] = O \big(  n^{-1} (\log n)^{-3} \big),
$ 
and therefore  by the condensation test,
$$
\sum_n
b_n^{-2} \E[Z_n^2]  < \infty.
$$
Hence by a consequence of the martingale convergence theorem
(see \cite{Feller}, Theorem VII.9.3, page 243) we have almost surely
that as $n \to \infty$,
$$
(\log n)^{-2}  \sum_{i = k+1}^n Z_i \to 0.
$$
Combined with \eqref{0116d}
and \eqref{0116b} this shows
that $N^*_{n} \sim (c/2) (\log n)^2$, a.s.
But also by \eqref{e:Enev}, $N_{n} - N_{n}^*$ is a.s.
bounded in $n$, so
 \eqref{0125bc} follows.

To obtain \eqref{0125cc}, let $\eps \in (0,1)$.
For each $m \in \N$ set
$n^+(m) := \lceil \exp( \sqrt{2m(1+\eps)/c} \, ) \rceil$
and
$n^-(m) := \lfloor \exp( \sqrt{2m(1-\eps)/c} \, ) \rfloor$.
Then by (\ref{0125bc}), almost surely 
we have for large enough $m$ that
\begin{align*}
	N_{n^+(m)} &  > (1+\eps)^{-1}  (c/2) (\log n^+(m))^2 \geq m;
\\
	N_{n^-(m)}  & \leq (1+\eps)  (c/2) (\log n^-(m))^2 < m,
\end{align*}
so that $n^-(m) \leq \nu_m \leq n^+(m)$, and hence
$$
(1- \eps)
\exp \big(\sqrt{2(1- \eps)/c} \big) 
\leq \nu_m^{1/\sqrt{m}} \leq
(1+ \eps) 
\exp \big( \sqrt{2(1+\eps)/c} \, \big)
.
$$
Since $\eps$ can be arbitrarily small, this yields
\eqref{0125cc}.
\end{proof}

 \section{Preliminaries}
 \allco

 \subsection{Further notation and terminology}
 \label{ss:notation}
We use $\lambda_d$ to denote $d$-dimensional Lebesgue measure.
Let $\mu$ denote the common probability distribution of $X_i, i \in \N$,
i.e. $\mu(dx) = f(x) dx$. Let $o:= (0,\ldots,0)$, the origin in $\R^d$.

 Given $n \in \N$, we write $[n]$ for $\{1,\ldots,n\}$.
 Given $I \subset [n]$, we write $\cX_I$ for $\{X_i: i \in I\}$
 (so $\cX_n = \cX_{[n]}$).
 We use $|\cdot|$ to denote the number of elements of a finite set.

For $i \in [d]$ let $e_i $ denote the $i$th coordinate vector
in $\R^d$
(so $e_1= (1,0,\ldots,0)$ and so on).
Let $\mathbb H_2$
denote the linear span of $e_1$ and $e_2$ (a two-dimensional subspace
of $\R^d$). For $x \in \mathbb H_2$ and $r \geq 0$, set $B_2(x,r):=
B(x,r) \cap \mathbb H_2$ (a two-dimensional ball/disk).

For distinct $x,y \in \R^d$ let $[x,y]$ denote the line segment from $x$ to $y$,
i.e. the convex hull of $\{x,y\}$.
 For non-empty bounded $D \subset \R^d$,  let
 $\dist(x,D):= \inf_{z \in D} \|z-x\|$, and
 $\diam(D) := \sup_{u,v \in D} \|u-v\|$.
 Also for $r>0$ let
	$\kappa(D,r)$ be the number of deterministic
	balls of radius $r$ needed
	to cover $D$, i.e.
	 $\kappa(D,r) := \min \{n \in \N: \exists x_1,\ldots,x_n
	\in \R^d$ with $D \subset \cup_{i=1}^n B(x_i,r)\}$.



Given $n,k \in \N$ and $r >0$,
	let $F_{n,k,r}$ denote the region covered at least $k$ times by the
	balls of radius $r$ centred on $X_1,\ldots,X_n$, i.e.
	$F_{n,k,r}:= \{x \in \R^d: \cX_n(B(x,r)) \geq k\}$.

 \subsection{Geometrical preliminaries}

\begin{defn}[Sphere condition]
	Suppose $\partial A \in C^{1,1}$.
        For $z \in \partial A$ let $\hat n_z$ be the unit normal to $\partial A$ at $z$ pointing inside $A$.

        Given $\tau \geq 0$,
        let us say $\tau $ satisfies the {\em sphere condition}
         for $A$ if, for all $x \in \partial A$,
        we have $B(x+ \tau \hat n_x , \tau) \subset A$
        and  $B(x -\tau \hat n_x, \tau) \cap A = \{x\}$.

Let $\tau(A)$ denote the supremum of the set of all $\tau$ satisfying
the sphere condition for $A$.
\end{defn}
\begin{lemm}[Sphere condition lemma]
	Suppose $\partial A \in C^{1,1}$. Then
        $\tau(A) >0$; that is,
        there exists a constant $\tau >0$ such that
        $\tau$ satisfies the sphere condition
         for $A$.
\end{lemm}
\begin{proof}
        See     \cite[Lemma 7]{sphere-condition-lp}.
\end{proof}

\begin{lemm}[Balls near the boundary]
	\label{l:bdyballs}
	Suppose $\partial A \in C^{1,1}$.  Let $\eps \in (0,1)$.  Then:
	
	(i) There exists $\delta >0$ such that
	for all $r \in (0,\delta)$
	and all $x \in A^{(\eps r)}$
	we have $\lambda_d(B(x,r) \cap A) \geq (1+(\eps/2)) \theta r^d/2$;
	
	(ii)  There exist $\delta' >0, \delta'' >0$ such that if
	$s \in (0, \delta')$ and
	$x \in A \setminus A^{(\delta''s)}$
	then $|\lambda_d(B(x,s) \cap A) -  (\theta/2) s^d|
	< \eps s^d$.

	(iii)
	There exists $\delta''' > 0$ such that if $r \in (0, \delta''')$
	and $x \in A$ then $\lambda_d(B(x,r) \cap A)
	\geq ((\theta/2) - \eps) r^d$.
\end{lemm}
	\begin{proof}
		For (i) and (iii) we use \cite[Lemma 2.3(i)]{PYconnect} 
		and \cite[Lemma 2.4(i)]{PYconnect} respectively.
		For (ii) we take $r \downarrow 0$ in
		\cite[Lemma 2.3(ii)]{PYconnect}
		for an upper bound
		on $\lambda_d(B(x,s) \cap A)$, and use (iii)
		for a lower bound.

		In the case $d=1$ we interpret the $C^{1,1}$ condition
		as meaning $A$ is a finite union of intervals, and
		then  (i), (ii) and (iii) all clearly hold.
	\end{proof}

\begin{lemm}[Cone condition lemma]
	\label{l:conecond}
	Suppose either $\partial A \in C^{1,1}$ or $A$ is convex.
	There exists  constants $\tau_1(A) \in (0,\frac12]$
	and $\delta_1 = \delta_1(A) >0$
	such that for all $x \in A$ and all $r \in (0,\tau_1(A))$,
	$\lambda_d(B(x,r) \cap A) \geq \delta_1 r^d$.
\end{lemm}
\begin{proof}
	Suppose $\partial A \in C^{1,1}$. In this case
	the result is immediate from Lemma \ref{l:bdyballs}(iii).

	Suppose instead that $A$ is convex. Choose $x_0 \in A$ and $\delta' >0$
	such that $B(x_0,\delta') \subset A$.
	Suppose
	$0 < r < \delta'$
	and $x \in A$. Let $A'$ be the convex hull of $\{x\} \cup 
	B(x_0,\delta')$. Then $A' \subset A$, and since we assume $A$ 
	is bounded, the set $B(x,r) \cap A'$ is the intersection of
	$B(x,r)$ with a cone with vertex $x$ having 
	aperture bounded below by some positive constant
	independent of $x$. Hence there is a constant $\delta''$
	independent of $x$  and $r$ such that $\lambda_d(B(x,r) \cap A)
	\geq \lambda_d(B(x,r) \cap A') \geq \delta'' r^d$. 
\end{proof}

\begin{lemm}[Covering boundary regions]
	\label{l:Covbdy}
	Suppose $\partial A \in C^{1,1}$ or $A$ is convex. Let $K >0$.
	Then 
	(i) $\kappa(\partial A,r) = O(r^{1-d})$ as $r \downarrow 0$, and
	(ii) $\kappa(A \setminus A^{(Kr)},r) = O(r^{1-d})$ as $r \downarrow 0$.
\end{lemm}
\begin{proof}
	If $A$ is convex, then (i) holds because
	 $\partial A$ has finite $(d-1)$-dimensional Hausdorff measure
	by the Steiner formula; see e.g. \cite{CSKM}, or 
	\cite[eq. (A.23)]{LP}.

	To prove (i) in the case when $\partial A \in C^{1,1}$,
	let $x \in \partial A$ and take bounded open $U \subset \R^{d}$,
	such that $\partial A \cap U$, after a rotation, is the graph
	of a Lipschitz-continuously differentiable function $\phi$ defined
	on an open set in $\R^{d-1}$. By a compactness argument, 
	it suffices to prove that $\kappa(\partial A \cap U,r) = O(r^{1-d})$.
	Without loss of generality, we can and do assume the rotation
	in question is the identity map. Thus to prove (i),
	it suffices  to prove that
	for any  open ball $V \subset \R^{d-1}$
	and Lipschitz-continuously differentiable
	$g:V \to \R$ we have
	\begin{align}
		\kappa (\{(x,g(x)): x \in V \},r) = O(r^{1-d})~~~~ {\rm as}~ r
		\downarrow 0.
		\label{e:Lipcov}
	\end{align}
	Consider such a $V$ and such a $g$.
	By the differentiability condition
 and the intermediate
	value theorem, there is a constant $M \geq 1$
	such that 
	$$
	|g(x) - g(y)| \leq M \|x-y\|, ~~~ x,y \in V.
	$$
	Cover $V$ by 
	balls of radius $r/(2M)$,
	with centres denoted $x_{r,1},\ldots,x_{r,m(r)}$ say,
	where $x_{r,i} \in V$ for each $i \in [m(r)]$, and
	with $m(r) = O(r^{1-d})$ as $r \downarrow 0$. 
	Then for any $x \in V$, choosing $i = i(x) \in [m(r)]$
	with $\|x- x_{r,i}\| \leq r/(2M)$, we have 
	$|g(x) -g(x_{r,i}) | \leq M\|x-x_{r,i}\| \leq r/2$ and hence
	$\|(x,g(x))- (x_{r,i},g(x_{r,i})) \| \leq (r/2) + r/(2M) \leq r$ 
	so the balls $B((x_{r,i},g(x_{r,i})),r), i \in [m(r)]$ 
	cover the set $\{(x,g(x)):x \in V\}$. Thus we have
	\eqref{e:Lipcov}, and hence (i).

	Using (i),  we can prove (ii)
	by following
	the proof of \cite[eq. (6.17)]{MPCov} with minor modifications.
\end{proof}
	\subsection{Probabilistic preliminaries}

	Recall the definitions of $R_{n,k}$ and $\tR_{n,k}$
	at \eqref{e:defRnk} and \eqref{eqmaxspac}.
	We shall use
 the following result from \cite[Proposition 4.9 and
 Theorem 4.1]{MPCov}.

\begin{lemm}[Almost sure asymptotics for $R_{n,k}$ and $\tR_{n,k}$]
	\label{lem1}
Let  $k \in \N$.
	Then
	\begin{align}
\lim_{n \to \infty} (n \theta  f_0 \tR_{n,k}^d /\log n) =  1,
~~~ a.s.
\label{0315b}
	\end{align}
If $B \subset A^o $ is compact
	with $\lambda_d(B) >0$ and $\lambda_d(\partial B)=0$,
	then almost surely,
	\begin{align}
\lim_{n \to \infty} (n \theta  f_0 R_{n,k}^d /\log n) =  1.
\label{0315a}
	\end{align}
	If $B=A$ and $\partial A \in C^{1,1} $ 
	then
	a.s.,
	\begin{align}
	\lim_{n \to \infty} (n \theta R_{n,k}^d/\log n) =
	\max( 1/f_0, (2-2/d)/f_1).
\label{0315c}
	\end{align}
\end{lemm}
The $C^{1,1}$  condition for \eqref{0315c} above is slightly weaker
than the $C^2$ condition imposed in the statement of \cite[Proposition 4.9]{MPCov}, but the proof there still works if we assume only that $\partial A \in C^{1,1}$.

We shall also  use the following lemma, which
is a special case of \cite[Lemma 6.1]{MPCov}.
It says
that if an array of random variables $U_{n,k}$ is
monotone in $n$ and $k$, and $n U_{n,\lfloor \eps \log n \rfloor}/\log n$,
is bounded above in probability for some fixed $\eps >0$,
one may be able to prove the same bound asymptotically almost surely
for $n U_{n,k}/\log n$ for any fixed $k$.
\begin{lemm}[Subsequence trick]
\label{lemtrick}
Suppose $(U_{n,k},(n,k) \in \N \times \N)$ is an array of random
variables on a common probability space such that $U_{n,k} $ is
nonincreasing in $n$ and nondecreasing in $k$, that is,
$U_{n+1,k} \leq U_{n,k} \leq U_{n,k+1} $ almost surely,
for all $(n,k) \in \N \times \N$. 

	Let $c > 0$,  $\eps >0$. Suppose
$
\Pr[n U_{n, \lfloor  \eps \log n \rfloor } > \log n] \leq
c n^{-\eps},
$
        for all but finitely many $n$.
	Let $k \in \mathbb N$.
Then with probability 1, $ \limsup_{n \to \infty} n U_{n,k} /\log n \leq 1$.

\end{lemm}

The following asymptotic upper bound on coverage thresholds comes from
taking $\beta =0$ in
\cite[Lemma 6.3]{MPCov}.

\begin{lemm}[General upper bound]
\label{lemmeta}
Suppose $r_0, a, b \in (0,\infty)$,  and
 a family of sets  $A_r \subset A,$ defined for $0 < r < r_0$,
        are
         such that  for all $r \in (0,r_0)$, $x \in A_r$ and
 $s \in (0,r)$ we have $\mu(B(x,s)) \geq a s^d$, and
moreover $\kappa(A_r,r) = O(r^{-b})$ as $r \downarrow 0$.


	Let $\alpha > b/(ad)
	$ and set $r_n = (\alpha (\log n)/n)^{1/d}$.
 Then there exists $\delta >0$ such that
as $n \to \infty$, we have
$
\Pr[\{A_{r_n} \subset F_{n,\lfloor \delta \log n \rfloor,r_n} \}^c ]
 = O(n^{-\delta})
$
\end{lemm}

We can combine the previous two lemmas to obtain the following.
\begin{lemm}
	\label{l:Sn}
Suppose $r_0, a, b \in (0,\infty)$,  and
 a family of sets  $A_r \subset A,$ defined for $r >0$,
        are
         such that  for all $r \in (0,r_0)$, $x \in A_r$ and
 $s \in (0,r)$ we have $\mu(B(x,s)) \geq a s^d$, and
moreover $\kappa(A_r,r) = O(r^{-b})$ as $r \downarrow 0$.

	For $n,k \in \N$ let $S_{n,k}:= \inf\{r >0: A_r \subset F_{n,k,r}\}.$
	Then
	$\limsup_{n \to \infty} 
	nS_{n,k}^d/\log n \leq b/(ad)$,
	almost surely,
	for any fixed $k \in \N$. 
\end{lemm}
\begin{proof}
	Observe that 
	$F_{n,k,r} \subset F_{n+1,k,r}$
	and
	$F_{n,k,r} \supset F_{n,k+1,r}$
	for each $(n,k,r)$, so that
	 $S_{n,k}$ is nonincreasing in $n$
	and nondecreasing in $k$ (regardless of whether $A_r$ is monotone
	in $r$).

	Let $\alpha > b/(ad)$ and set $r_n := ( \alpha (\log n)/n)^{1/d}$.
	By Lemma \ref{lemmeta} there exists $\delta > 0$
	such that for $n \geq e^{k/\delta}$,
	$$
	\Pr[n S_{n,k}^d > \alpha \log n] = \Pr[S_{n,k} > r_n]
	\leq \Pr[\{ A_{r_n} \subset F_{n, \lfloor \delta \log n \rfloor,r_n}
	\}^c] 
	= O(n^{-\delta}),
	$$
	and hence by Lemma \ref{lemtrick}, 
	almost surely $\limsup_{n \to \infty} (n S_{n,k}^d/\log n) \leq \alpha$.
	Taking $\alpha \downarrow b/(ad)$ gives us the result.
\end{proof}

For $k,n \in \N$ and 
$D \subset A$ let $\tR_{n,k} (D) := \inf \{r > 0: D \cap A^{(r)}
\subset F_{n,k,r}\}$
(so $\tR_{n,k} = \tR_{n,k}(A)$).
The next lemma gives asymptotic upper bounds
for $R_{n,k}(D)$
and $\tR_{n,k}(D)$
for arbitrary compact $D \subset A$,
including the intermediate case with $D \neq A$ and $D$  not 
contained in $A^o$. We use the conventions $\inf(\emptyset) := +\infty$
and $1/+\infty :=0$.
\begin{lemm}
	\label{l:upper}
	Let $D \subset A$ be compact and non-empty.
	Set $f_0(D):= \inf_{x \in D} f(x)$ and set $f_1(D) :=
\inf_{x \in D \cap \partial A} f(x)$. 
		Assume $f_0(D) >0$.
	Let $k \in \N$. Then, almost surely,
	\begin{align}
		\limsup_{n\to \infty} \Big( n \theta \tR_{n,k}(D)^d/\log n
		\Big)  \leq 1/f_0(D), 
		\label{e:upper2}
	\end{align}
	and if $\partial A \in C^{1,1}$,
	\begin{align}
	\limsup_{n \to \infty}
		\Big( n \theta R_{n,k}(D)^d/\log n \Big)  
		\leq \max(1/f_0(D), (2-2/d)/f_1(D)). 
		\label{e:upper1}
	\end{align}
\end{lemm}
	\begin{proof}
		Set $f'_0:= f_0(D)$ and $f'_1:= f_1(D)$.
		Let $\eps' \in (0,f'_0)$. For all small enough $r>0$, all
		$x \in D \cap A^{(r)}$
		and $s \in (0,r)$ we have
		$\mu(B(x,s)) \geq (f'_0- \eps') \theta s^d$. By
		 Lemma \ref{l:Sn} with  $A_r = D \cap A^{(r)}$,
		 $a = (f'_0 - \eps')\theta$ and $b=d$, we obtain that
		 almost surely
		 $$
		 \limsup_{n \to \infty} n \tR_{n,k}(D)^d/\log n 
		 \leq 1/(\theta (f'_0 - \eps')),
		 $$
		 and taking $\eps' \downarrow 0$ yields \eqref{e:upper2}.

		Now assume $\partial A \in C^{1,1}$.
		Also assume for now that $D \cap \partial A \neq \emptyset$
		so $f'_1 < \infty$.
		Let $\beta > \max(1/(f'_0 \theta),
		(2-2/d)/(f'_1 \theta))$.
		Let	$\eps \in (0,f'_0)$
		with $1/((f_0'-  \eps)\theta) < \beta$
		and $(2-2/d) /((f'_1 -  \eps)\theta) < \beta$.
		Let $r_n := (\beta (\log n)/n)^{1/d}$.

		By Lemma \ref{l:Covbdy},
	$\kappa(D \setminus A^{(r)},r) = O(r^{1-d})$
	as $r \downarrow 0$. Also by Lemma
		\ref{l:bdyballs}(iii)
	and the continuity of $f$,
		for all small enough $r >0$, 
		all $x \in D \setminus A^{(r)}$ and all
		$s \in (0,r)$,
		$\mu(B(x,s))
	> \frac12 (f'_1-  \eps) \theta s^d$.  
	By applying Lemma \ref{lemmeta}
	taking $A_r = D \setminus A^{(r)}$ with $a = (f'_1- \eps) \theta/2$
		and $b = d-1$ so that $b/(ad) < \beta$,
		we obtain for some $\delta >0$ that 
		as $n \to \infty$,
	\begin{align}
	\Pr[\{D \setminus A^{(r_n)} \subset F_{n,\lfloor \delta \log n \rfloor, r_n} \}^c ]  = O(n^{-\delta}).
		\label{e:D1}
	\end{align}
	For all small enough $r >0$,  all
		$x \in D \cap A^{(r)} $ and all $s \in (0,r)$
		we have $\mu(B(x,s)) \geq (f'_0- \eps) \theta s^d$,
	and $\kappa (D \cap A^{(r)},r) = O(r^d)$ as $r \downarrow 0$.
	By applying Lemma \ref{lemmeta} to $A_r = D \cap A^{(r)}$
		with $a = (f'_0 - \eps) \theta$ and $b=d$
		(so $b/(ad)  < \beta$), we can find $\delta' \in (0,\delta]$
	such that
	\begin{align}
	\Pr[\{D \cap A^{(r_n)} \subset F_{n,\lfloor \delta' \log n \rfloor, r_n} \}^c ]  = O(n^{-\delta'}).
		\label{e:D2}
\end{align}
	Set $R_{n,k}(D) := \inf\{r: D \subset F_{n,k,r}\}.$ 
	By \eqref{e:D1}, \eqref{e:D2}
		and the union bound,
	\begin{align*}
		\mathbb P[n R_{n,\lfloor \delta' \log n \rfloor}(D)^d /\log n 
		> \beta]
		& = \mathbb P[ R_{n,\lfloor \delta' \log n \rfloor}(D) > r_n ]
		\\
		& = \mathbb P[\{D \subset F_{n,\lfloor \delta' \log n
		\rfloor,r_n}\}^c ] = O(n^{-\delta'}),
	\end{align*}
	so by Lemma \ref{lemtrick}, almost surely 
		we have $ \limsup_{n \to \infty}
	n  R_{n,k}(D)^d/\log n \leq \beta$.
		Since $\beta $ is arbitrary subject to
		 $\beta > \max(1/(f'_0 \theta), (2-2/d)/(f'_1 \theta))$,
		 this gives us \eqref{e:upper1} when $D \cap \partial A 
		 \neq \emptyset$.
		 
		 If $D \cap \partial A = \emptyset$, we can still obtain
		 \eqref{e:upper1} by repeating the preceding argument,
		 now taking $\beta, \varepsilon$ just with
		  $\beta > 1/((f'_0 -\eps)\theta)$. In this case \eqref{e:D1}
		 is automatic because $D \setminus A^{(r_n)} = \emptyset$
		 for $n$ large.
\end{proof}

Next we give a further preliminary 
result providing a uniform exponential moment bound
on the random variables $nR_{n,k}^d/\log n$ and $n \tR_{n,k}^d/\log n$. 
We shall use this to check the uniform integrability condition of Lemma
\ref{l:genrec}
in our situations.

\begin{lemm}
\label{meanlem}
	For $n \in \N$, set $r_n:= ((\log n)/ n )^{1/d}   $.
	Let $k \in \mathbb N$. 

	(i) Suppose  either (a) that
	$B \subset A^o$  or (b) that 
	$B=A$ and either $\partial A \in C^{1,1}$ or $A$ is convex (or both). 
	There exists a constant $\delta_2 = \delta_2(d,A) >0$ such that
	\begin{align}
	\limsup_{n \to \infty} \E[ \exp(\delta_2  (R_{n,k}/r_n)^d)] < \infty.
\label{eqmeanlem}
	\end{align}
	(ii) Suppose that  $\partial A \in C^{1,1}$ or $A$ is convex.
	Then there exists $\delta_3 = \delta_3(d,A) > 0$ such that 
	\begin{align}
	\limsup_{n \to \infty} \E[ \exp(\delta_3  (\tR_{n,k}/r_n)^d)] < \infty.
\label{eqmeanlem2}
	\end{align}
\end{lemm}
\begin{proof}
	(i) 
	Using the continuity of $f$ and (in case (b)) 
	Lemma \ref{l:conecond}, we can (and do) choose
	$\delta, \delta' \in ( 0, \frac12 \min(1, 1/f_0))$ such that 
	for all $x \in B $ and $r \in (0,\delta)$ we have
	$\lambda_d(B(x,r) \cap A) \geq 2 \delta' r^d$ and
	$\inf_{x \in B(x,r) \cap A} f(x) \geq f_0/2$.
	Then in both cases, for all $x \in B$ and $r \in (0,\delta)$,
	we have $\mu(B(x,r)) \geq \delta' f_0 r^d$.

	Given $n$, partition  $\R^d$ into cubes  
	of side $r_n/(2d)$. Let those cubes in the partition  having
	non-empty intersection with   
	$B$ be denoted
	$Q_{n,1},\ldots, Q_{n,m_n}$, and for each $i \in [m_n]$
	let $q_{n,i}$ denote the first point of $\overline{Q_{n,i}} 
	\cap B$ in the lexicographic ordering.

Let $t \geq 1$.
	Suppose $\X_n ( B(q_{n,j}, (t- \frac12) r_n)) \geq k$
for  each $j \in [m_n]$.
	Then for  each $x \in \cup_{j \in [m_n]} Q_{n,j}$
	we have
	for some $j \in [m_n]$ that
	$\|x-q_{n,j} \| \leq r_n/2$,
	so by the triangle inequality
	$B(q_{n,j},(t- \frac12) r_n) \subset B(x,t r_n)$ and hence
	$\X_n(B(x,t r_n)) \geq k$
	so that $R_{n,k} \leq t r_n$. Thus,
there exist constants $c >0$ and $n_0 \in \N$,
 independent of $t$, such that
for all $n \geq n_0$ and all $t$ with
 $t r_n < \delta$ and $t \geq 1$, 
	\begin{align*}
	\Pr[ R_{n,k} > t r_n] & \leq \Pr [ \cup_{j=1}^{m_n} 
	\{ \X_n ( B(q_{n,j},(t- \tfrac12) r_n )) < k  \} ]
\\
	& \leq c (n/\log n) \binom{n}{k-1}
	\big(1 -  \delta' f_0 (t - \tfrac12 )^d  r_n^d \big)^{n-k+1}
\\
	& \leq n^k \exp (- n \delta' f_0 (t- \tfrac12)^d r_n^d) 
\\
		& =  n^{k- \delta' f_0(t-1/2)^d} \leq n^{k- \delta' f_0(t/2)^d}.
		\end{align*}
	Hence, given $a \in ( 0,1]$, for large enough $n$ we have that
	\begin{align}
	\E[ \exp(a (R_{n,k}/r_n)^d ) {\bf 1}_{\{R_{n,k} < \delta \}} ]
		& \leq \int_0^\infty \Pr[ e^{a (R_{n,k}/r_n)^d} > s; R_{n,k} \leq \delta] ds
\nonumber \\
		& \leq e +  \int_e^\infty \Pr[  R_{n,k} > r_n( (\log s)/a)^{1/d}; R_{n,k} \leq \delta] ds
\nonumber \\
		& \leq e+  n^k \int_e^\infty \exp(   - \delta' 
		f_0 (\log n) (\log s) / (2^{d} a )) ds
\nonumber \\
		& \leq e+  n^k \int_e^\infty s^{ - \delta' f_0 (\log n) / (2^{d} a )}  ds
\nonumber \\
		& = e + \frac{n^k e^{1- \delta' f_0 (\log n)/(2^{d} a) } }{
			(\delta' f_0 \log n/(2^{d} a)) -1 } 
\nonumber \\
		& \leq e + n^{ k - \delta' f_0 /(2^{d} a)} \leq 3, 
\label{0308a}
	\end{align}
provided $a$ is chosen so $2^{d} a k  < \delta' f_0 $.

Also, by our assumptions on $f$ and $B$ we can cover $B$ by 
 a finite collection
of balls of radius at most $\delta/3$, denoted $B_1,\ldots,B_\ell$ say,
and find a constant $\delta'' \in (0,\frac12)$,
such that $\mu(B_i) \geq \delta''$ for each $i \in [\ell]$. Then 
\begin{align*}
\Pr[R_{n,k}  \geq \delta] \leq \Pr[ \cup_{i=1}^\ell
\{ \X_n ( B_i) < k \} ] & \leq  \ell \binom{n}{k-1} 
(1- \delta'')^{n-k+1} 
\\
& \leq   \ell 2^{k-1} n^{k-1} \exp(-n \delta'').
\end{align*}  
Since also $R_{n,k} \leq \diam(A)$, 
we have
	\begin{align*}
\E[ \exp(a (R_{n,k}/r_n)^d ) {\bf 1}_{\{R_{n,k} \geq \delta \}} ]
\leq 
		e^{a (\diam (A))^d (n / \log n) } \ell 2^{k-1} 
		n^{k-1} \exp(-n \delta''),
	\end{align*}
which tends to zero as $n \to \infty$.
Combined with (\ref{0308a}), this gives us part (i).

For (ii), observe that $\tR_{n,k} \leq R_{n,k}$
(for $B=A$) and therefore taking $\delta_3 = \delta_2$,
we obtain \eqref{eqmeanlem2} from the case $B=A$ of 
\eqref{eqmeanlem}.
\end{proof}

\section{Uniqueness of the furthest location from $\X_n$}
\label{s:uniqueness}
\allco

For $x \in \R^d$ and $n,k \in \N$ with $n \geq k$, define $k$-$\dist(x,\X_n)
: = \inf \{r \geq 0: \X_n(  B(x,r)) \geq k\}$,
the distance from the location $x$ to its $k$-nearest neighbour in $\X_n$
(we generally use `location' to denote a generic element of $\R^d$, and
`point' to denote a point of the point process $\X_n$).
Given $A$ and $B$, define the random sets
\begin{align}
	\cW_{n,k} &:= \cW_{n,k}(B) :=
	\{x \in B: k{\mbox -}\dist(x, \cX_n) = R_{n,k}\};
	\label{e:cWdef} \\
	\tcW_{n,k} &:= \{x \in A^{(\tR_{n,k})}:
	k{\mbox -}\dist(x, \cX_n) \geq \tR_{n,k}\}
	.
	\label{e:tcWdef}
	\end{align}
	By \eqref{e:defRnk}, $R_{n,k} = \inf\{r: k\mbox{-}\dist(
	x,\X_n) \leq r \: \forall x \in B\}
	= \sup\{k\mbox{-}\dist(x,\X_n):x \in B\}$.
Since $k$-$\dist(x,\cX_n)$ is continuous in $x$ and $B$ 
is compact, ${\cal W}_{n,k} \neq \emptyset$.
	Let $W_{n,k} $ be the first element of
	${\cal W}_{n,k}$ in the lexicographic ordering.

	By \eqref{eqmaxspac}, and the definition of $F_{n,k,r}$ in
	Section \ref{ss:notation},
	\begin{align}
		\tR_{n,k} =
	\sup\{r:  A^{(r)} \setminus F_{n,k,r} \neq \emptyset \}.
		\label{e:tRdef2}
	\end{align}
	By a subsequence argument there exists $x \in A^{(\tR_{n,k})}$
	with $k$-$\dist(x,\X_n) \geq \tR_{n,k}$, so $\tW_{n,k} \neq \emptyset$.
	Let $\tW_{n,k} $ be the first element of
	$\tcW_{n,k}$ in the lexicographic ordering,

As mentioned earlier, we would like the furthest location in $B$
from the sample $\X_n$ (or for general $k$, the
location in $B$ at the greatest $k$-distance from $\X_n$) to
be unique: in other words, we need $|\cW_{n,k}|=1$.  
As also 
mentioned, proving that this uniqueness holds almost surely seems to
be harder than one might expect, especially when $B =A$, 
and in fact we prove only that uniqueness holds almost surely
`for large enough $n$' rather than `for all $n$'.
This is sufficient for our purposes. 

When considering records for the maximal $k$-spacing,
we would similarly like the location in $A^{(R_{n,k})}$
at the greatest $k$-distance from $\X_n$ to be unique,
i.e. $|\tcW_{n,k}|=1$. We shall prove
likewise that this uniqueness holds almost surely for large enough $n$.
We now state our uniqueness results.
\begin{prop}
	[Uniqueness in the interior]
\label{lemWn}
Let $n,k \in \N$ with $ n \geq k$. Then, almost surely,
	$|\cW_{n,k} \cap B^o| \leq 1$ and 
	$|\tcW_{n,k} \cap (A^{(\tR_{n,k})})^o| \leq 1$.
\end{prop}

The next result 
 says that there is a constant $\delta_4  >0$
 such that almost surely the location in
$A$ at greatest $k$-distance from the sample $\cX_n$ is unique, provided 
$R_{n,k} \leq \delta_4 $.

\begin{prop}[Main uniqueness result]
	\label{p:uniquefurth}
	Suppose $B=A$ and $\partial A \in C^{1,1}$.
	There exists a constant $\delta_4 = \delta_4(d,A) > 0 $ such that
	for all $k, n \in \N$
	with $k \leq n$, we have $\Pr[
		\{ |\cW_{n,k}| > 1 \} \cap \{R_{n,k} \leq \delta_4\}] = 0$.
\end{prop}

We now start to prepare for proving Propositions \ref{lemWn}
and \ref{p:uniquefurth}.
The  next lemma is deterministic: there
is no `almost surely' qualification in its conclusion.
To save space we consider 
$\cW_{n,k}$ and $\tcW_{n,k}$ together in this lemma.

\begin{lemm}
	\label{l:I>d}
	Let $n,k \in \N$ with $n \geq k$. 
	Either (a) let $R= R_{n,k}$
	and let $y \in \cW_{n,k} \cap B^o$, 
	or (b) let $R= \tR_{n,k}$ and let 
	$y \in \tcW_{n,k} \cap (A^{(\tR_{n,k})})^o$.
	In either case
	define the index sets  
	\begin{align}
		I 
		:= \{i \in [n] : \|y -X_i\| = R\}; 
	~~~~~~~~
		I_0 
		:= \{i \in [n]: \|y -X_i\| < R\}.
		\label{e:defII0}
		\end{align}
		Then $|I| \geq d+1$.
%
\end{lemm}
\begin{proof}
	We prove the result by contradiction, so
	suppose $|I| \leq d$.

	We claim that $k$-$\dist(y,\X_n) =R$. In case (a), this
	is immediate from the definition \eqref{e:cWdef}. In case (b),
	by \eqref{e:tcWdef} we have
	$k$-$\dist(y,\X_n) \geq R = \tR_{n,k}$, and if the inequality is
	strict, then 
	 there exists $s > \tR_{n,k}$ such that
	$y \in A^{(s)}$ and
	$k$-$\dist(y,\X_n) > s$, contradicting \eqref{e:tRdef2}.

	By the preceding claim,
	$|I \cup I_0| \geq k$ and $|I_0| < k$ so $|I| \geq 1$.
	Choose $j \in I$.

	Let $H$ denote the intersection of the hyperplanes tangent to
	$B(X_i,R)$  at $y$, $ i \in I \setminus \{j\}$ (or if $|I|=1$,
	let $H := \mathbb R^d$). Then
	$H$ is an affine subspace of $\mathbb R^d$ of dimension at least 1.
	Pick a unit vector $e \in \mathbb R^d$ such that
	the line  $L:= \{y + ae: a \in \R\}$  is contained in $H$.
	Let $L^+:= \{y+ae: a > 0\}$ and $L^-:= \{y+a e: a < 0\}$.

	Then $y \in L \cap \partial B(X_j,R)$, so 
	at least one of $L^+$ and $L^-$ (say $L^+$) 
	has empty intersection with $B(X_j,R)$. Since
	$L$ is tangent to each of $B(X_i,R), i \in I \setminus \{j\}$,
	moreover $L^+$ has empty intersection with each of these balls too.

	Since $y \notin \cup_{i\in [n] \setminus I \setminus I_0}
	B(X_i,R)$ we can and do choose $\delta >0$ with 
	$B(y, \delta) \cap \cup_{i \in [n] \setminus I \setminus I_0} B(X_i,R) = \emptyset$, 
	and moreover  with
	$B(y, \delta ) \subset B$ (in  case (a)) 
	or with $B(y, \delta) \subset A^{(R + \delta)}$
	(in case (b)).
	Then
	 $k$-$\dist(y + \delta e, \X_n) > R$,
	and $y + \delta e \in B$ (in case (a)) or $y + \delta e
	\in A^{(R + \delta)}$ (in case (b)),
	contradicting the definition of $R$ in either case.
	This completes the proof.
\end{proof}
We use the following notation: for each $i,j \in [n]$ with $i \neq j$, 
define the hyperplane
$$
P_{\{i,j\}} = \{x \in \R^d: \|x-X_i \|= \|x-X_j\| \},
$$
and for each $I \subset [n]$ with $|I| \geq 2$, set
\begin{align}
P_I = \cap_{\{i,j\} \subset  I } P_{\{i,j\}},
	\label{e:defPI}
\end{align}
the set of all locations in $\R^d$ that are equidistant from each $X_i,i \in I$.

Let $D_n$ denote the event that for each $I \subset [n]$ with $ 2 \leq |I| 
\leq d+1$,
the set $P_I$ is an affine subspace of $\R^d$ of dimension $d-|I|+1$.
Then $\Pr[D_n] = 1$. For example, if $I= [j]$ with $2 \leq j \leq d+1$,
then $P_I = P_{\{1,2\}} \cap P_{\{2,3\}} \cap \cdots \cap P_{\{j-1,j\}}$ which
is almost surely an affine space of dimension $d-(j-1)$.

If event $D_n$ occurs, then for each $I \subset [n]$ with $|I|= d+1$, the
set $P_I$ is an affine space of dimension zero, i.e. a single point. Denote
this point by $Y_I$.

\begin{proof}[Proof of Proposition \ref{lemWn}]
	We refer to the event that
	 $|\cW_{n,k} \cap  B^o| > 1 $
	 as `Case (a)' and the event that
	 $|\tcW_{n,k} \cap (A^{(\tR_{n,k})})^o| > 1 $
	 as `Case (b)'.
	In Case (a) we can pick distinct $y,z \in \cW_{n,k} \cap B^o$,
	while in Case (b) we can pick distinct $y,z \in \tcW_{n,k} \cap 
	(A^{(\tR_{n,k})})^o$.
	By Lemma \ref{l:I>d}, in either case there are index sets
	$I,J \subset [n]$, each with $d+1$ elements, such that 
	   $\|y-X_i\|, i \in I$ and $\|z- X_j\|, j \in J$ are all equal
	  to $R$, where we set $R = R_{n,k}$ in
	  Case (a) and $R = \tR_{n,k}$ in Case (b).
	  If also $D_n$ occurs, then $y = Y_I$ and $z = Y_{J}$, so
	  $I \neq J$. Hence
	  \begin{align}
		  ( \{|\cW_{n,k} \cap B^o | >1 \} 
		  &
		  \cup 
		  \{|\tcW_{n,k} \cap (A^{(\tR_{n,k})})^o | >1 \} ) 
	  \cap D_n 
		  \nonumber \\
		  &
		  \subset \cup_{I,J \subset [n]: |I|= |J|= d+1, I \neq J}
	  \{\dist(Y_J,\cX_J) = \dist(Y_I,\cX_I)\}.
		  \label{e:twoYs}
	  \end{align}
	  Let
	  $I,J \subset [n]$ with $ |I|= |J|= d+1, I \neq J$.
	  Let $j \in J \setminus I$.
	  For any  outcome of $(X_i: i \in [n] \setminus \{j\})$
	  that is consistent with event $D_n$, the set $P_{J \setminus \{j\}}$
	  is a line, and the set 
	  $$
	 V:=  \{x \in P_{J \setminus \{j\} }:  \dist(x, \cX_{J \setminus \{j\}})
	  = \dist (Y_I, \cX_I) \}
	  $$
	  has at most two elements.
	  Then, given the outcomes
	  of $(X_i,i \in [n] \setminus \{j\})$, the event 
	  $\{\dist(Y_J,\cX_J) = \dist (Y_I,\cX_I)\}$
	  occurs  only if 
	  $\|X_j -u \| = \dist(Y_I, \cX_I)$ for some $u \in V$,
	  an event of probability zero.
	  Then the result follows
	  from \eqref{e:twoYs}.
\end{proof}

\begin{lemm}[No location in $\partial A$ is equidistant  from $d+1$
	points of $\X_n$]
	\label{l:dAvtx}
	Let $n > d$ and
let $D'_n$ be the event that for each subset $I \subset [n]$
with $|I| = d+1$, the set $P_I \cap \partial A $ is empty.
	If $\lambda_d(\partial A) = 0$,
then $\mathbb P[D'_n] = 1$.
%
\end{lemm}
\begin{proof} 
	Consider just the case $I= [d+1]$. For $(\lambda_d)^{d+1}$-almost
	all values of $(x_1,\ldots,x_{d+1}) \in (\R^d)^{d+1}$,
	the hyperplanes $P_{\{i,j\}} , 1 \leq i < j \leq d+1 $ intersect in
	a single point denoted $p(x_1,\ldots,x_{d+1})$. Moreover
	$p(x_1,\ldots,x_{d+1} ) = x_1 + p(o,x_2-x_1, \ldots,x_d-x_1)$.
	Changing variables to $y_i = x_i -x_1$ for $i =2,\ldots,d$ we obtain
	whenever $\lambda_d(\partial A) = 0$ that
	\begin{align*}
		\int_{(\R^d)^{d+1}} & {\bf 1}_{\{p(x_1,\ldots,x_{d+1})
		\in \partial A 
		\}}
		d(x_1,\ldots,x_{d+1})
		\\
		= 
		& \int_{(\R^d)^d}
		\int_{\R^d}
		{\bf 1}_{\{x_1 + p(o,y_2,\ldots,y_{d+1}) \in \partial A\}}
		dx_1 d(y_2,\ldots,y_d)
		=0,
	\end{align*}
	and the result follows.
\end{proof} 
\begin{lemm}[$\partial(A^{(r)})$ is Lebesgue-null for $r$ small]
	\label{l:fromHPY}
	Suppose $\partial A \in C^{1,1}$. Then there exists $\delta_5 
	= \delta_5(d,A) \in (0,\tau(A)/9)$ such that for all
	$ r \in (0,\delta_5)$ we have  $\lambda_d(\{x \in A:
	\dist(x,\partial A) = r\}) =0$.
\end{lemm}
\begin{proof}
	Let $x \in \partial A$ and 
	choose $\delta (x) < \tau(A)/9$ such that after a rotation 
	$\cal R$ about $x$, within the ball
	$B(x,3 \delta(x))^o$ the set $A$ coincides with the closed epigraph of a $C^{1,1}$ function $\phi:U \to \R$, with $U$ an open ball of radius 
	$3 \delta(x)$ in $\R^{d-1}$ centred on $\pi(x)$, where $\pi: \R^d \to \R^{d-1}$ denotes projection onto the first $d-1$ coordinates. That is,
	$$
	\mathcal R(A) \cap B(x,3 \delta(x)) =
	\{(u,s): u \in U, s \in [\phi(u), \infty) \} \cap B(x,3 \delta(x)).
	$$
	By a compactness argument, it suffices to show that 
	if $r \in (0,\delta (x))$ then
	\begin{align}
	\lambda_d(\{y \in A \cap B(x,2\delta(x)): \dist(y,\partial A ) = r\})
	=0.
		\label{e:local}
		\end{align}
	By the argument in \cite{HPY} leading up to \cite[eqn (3.12)]{HPY},
	for bounded measurable $\psi:A \to [0,\infty)$
	supported  by $B(x,2\delta(x))$ and for $0 < s < \delta(x)$ we have
	$$
	\int_{A \setminus A^{(s)}} \psi(y) dy
	= \int_{U \times[0,s)} \psi(g(u,t))
	\Big| \det \Big( \frac{\partial g(u,t)}{\partial (u,t)} \Big) \Big|
	d(u,t),
	$$
	where for $(u,t) \in U \times (0,\delta(x))$ we set 
	$g(u,t):= (u, \phi(u)) + t \hat{n}_{(u,\phi(u))}.$ 
	Taking $r < s$ and
	setting $\psi(y) = {\bf 1}_{\{\dist(y,\partial A) =r\}}
	{\bf 1}_{\{\|y-x\|\leq 2 \delta(x)\}}$ 
	gives us $\psi(g(u,t)) \leq {\bf 1}_{\{t =r\}}$
	so that the integral is zero, and hence \eqref{e:local}
	as required.
\end{proof}
\begin{lemm}
	\label{l:I>=d}
	Suppose $B=A$ and $\partial A \in C^{1,1}$. Let $k,n \in \N$ with
	$k < n$. Let $y \in \cW_{n,k}$.
	If $R_{n,k} \leq \tau(A)/9$ then $|I| \geq d$, where
	$I,I_0$ are defined by taking $R= R_{n,k}$ in \eqref{e:defII0}.
\end{lemm}
\begin{proof}
	Suppose (to get a contradiction) that $|I| \leq d-1$.
	Then by Lemma \ref{l:I>d}, $y \in \partial A$.

	The intersection over all $i \in I$ of the hyperplanes tangent to
	$B(X_i,R_{n,k})$ at $y$ is an affine space of 
	dimension at least 1. Let $e \in \R^d$ be a unit vector such
	that the line $L:= \{y+ae: a \in \R\}$ is tangent to all
	of the balls $B(X_i,R_{n,k}), i \in I$.
	Assume without loss of generality that $e \cdot \hat{n}_y \geq 0$
	(otherwise replace $e$ with $-e$).

	 If $e \cdot \hat{n}_y >0$, then
	 by the sphere condition
	 there exists $\delta >0$ such that 
	 $y+ \delta e \in  A \setminus
	 \cap_{i \in [n] \setminus I \setminus I_0} B(X_i,R_{n,k}) $.
	 Then  $y +  \delta e \in A \setminus \cup_{i \in [n] \setminus I_0} B(X_i,R_{n,k})$ and $|I_0| <k$, so that $k$-$\dist(y+  \delta e,\X_n)
	 > R_{n,k}$, contradicting the definition of $R_{n,k}$.

	 Suppose $e \cdot \hat{n}_y =0$. Without loss of generality,
	 assume $y=o$, 
	 $e = e_1$ and $\hat{n}_y =e_2$. Set $\tau := \tau(A)/2$.
	 Recall $\mathbb H_2:=$span$(e_1,e_2)$.
	 For each $i \in I$, the intersection of $B(X_i,R_{n,k})$
	 with $\mathbb H_2$
	 is a two-dimensional
	 disk with radius at most $R_{n,k}$ with $o$ on its boundary
	 and with the $x$-axis tangent to it. Therefore
	 if $R_{n,k} \leq \tau(A)/9 < \tau/3$, there exist
	 locations in $\mathbb H_2 \cap B(\tau e_2, \tau) \setminus
	 \cup_{i\in I} B(X_i,R_{n,k})$ arbitrarily close to $o$;
	 see Figure \ref{f:tangents} (left). Hence
	  we can choose such a location $z$ say, lying outside
	 $\cup_{i \in [n] \setminus I \setminus I_0} B(X_i,R_{n,k})$.
	 Then by the sphere condition 
	 $z \in A$, and  $k$-$\dist(z, \X_n) > R_{n,k}$,
	 again contradicting the definition of $R_{n,k}$.
\end{proof}

\begin{figure}[p]
\center
\includegraphics[width=0.49\textwidth, trim= 0 10 5 20]{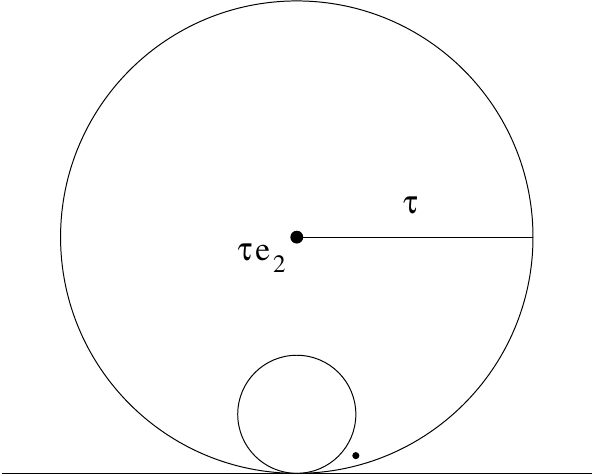}
\includegraphics[width=0.49\textwidth, trim= 0 10 5 20]{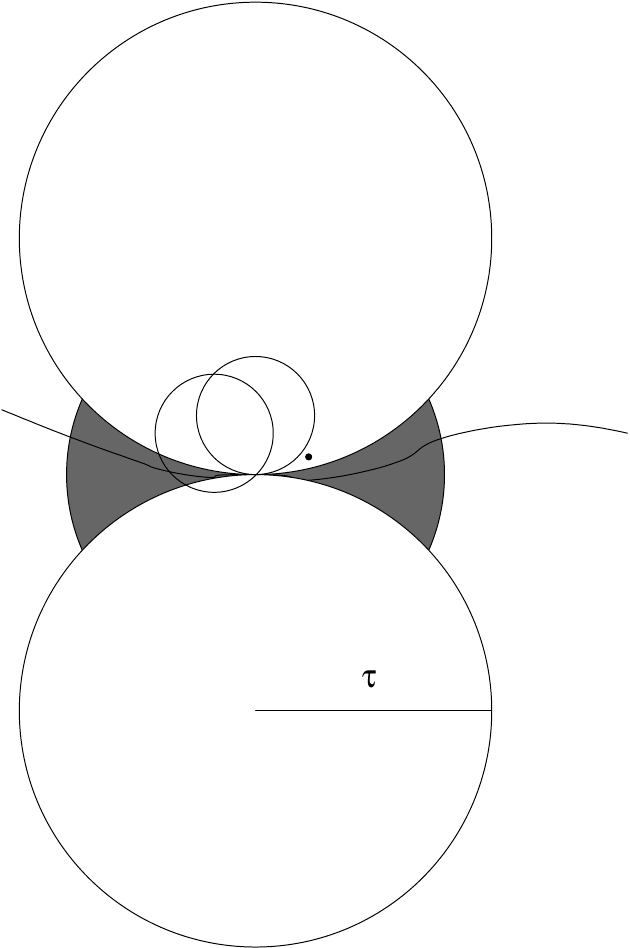}
\caption{\label{f:tangents}
	Left: In the proof of Lemma \ref{l:I>=d},
	there are locations in $\bH_2 \cap B(\tau e_2,\tau) \setminus
	\cup_{i \in I} B(X_i,R_{n,k})$ arbitrarily close to $o$ (the small dot shows one such location).
	Right: In the proof of Lemma 
	\ref{l:Hn},
	the curve is part of $\partial A \cap \bH_2$: near $o$ it is contained
	in the shaded region. The left shaded region is 
	$W^-$. The left-most small circle
	is  $\partial B(X_i,R_{n,k}) \cap \bH_2$, and passes through
	$o $ and $q'$; it is a rotation about $o$
	of the other small circle. The dot is a possible location for the point
	$y$ which lies in
	$A \cap \mathbb H_2 \setminus B(X_i,R_{n,k})$.}
\end{figure}

\begin{lemm}
	\label{l:Hn}
	Suppose $B=A$ and $\partial A \in C^{1,1}$.
	Let $k,n \in \N$ with $n > \max(k,d-1)$.
	  Let $I \subset [n]$ with $|I|=d$, and let
	$H_{n,I}$ be the event that
	  there exist two distinct locations $y, y' \in  
	  \partial A $, both at distance exactly $R_{n,k}$ from each of
	  $\{X_i, i \in I\}$.  Then 
	$\Pr[H_{n,I} \cap \{R_{n,k} \leq \tau(A)/9\}] = 0$.
\end{lemm}
\begin{proof}
	Without loss of generality, assume $1 \in I$. 
	If $D_n$ occurs, then  the set $P_I$ is a line, 
and if $H_{n,I} \cap D_n$ occurs then both points of $P_I$ at distance $R_{n,k}$
from $X_1$ (denoted $q$ and $q'$, say) 
lie in $\partial A$.
If moreover the event $D'_n$ 
	(defined in Lemma \ref{l:dAvtx})
	occurs, then
$k$-$\dist (q,\X_n \setminus \X_I) > R_{n,k}$ and
$k$-$\dist (q',\X_n \setminus \X_I) > R_{n,k}$.
We now show that if $R_{n,k} \leq \tau(A)/9$, then this cannot happen. 

Without loss of generality,  we assume $q=o$ and $\hat{n}_q = e_2$,
and moreover that $q' \in \mathbb H_2$.
	Let $I_0 := \{i \in [n]: \|X_i\| < R_{n,k}\}.$ Then
$|I_0 | < k$.

Let $\tau := \frac12 \tau(A)$.
          By the sphere condition
	  $B(\tau e_2,\tau) \subset A$ and $B(-\tau e_2, \tau) \cap A
	  = \{o\}$. Therefore 
          $$
	  \partial A 
	  \cap (B(\tau e_2,\tau)^o
          \cup B(-\tau e_2,\tau)^o) = \varnothing.
          $$

Suppose $q' \cdot e_1 \leq 0$ (the case 
 $q' \cdot e_1 \geq 0$ is treated analogously).
	  Since $q' \in   \partial A  \cap B(o,2R_{n,k}) $,
	  we have $q' \in W^-$, 
	  where (see Figure \ref{f:tangents}
	  (right)) we define the wedge-shaped set
          \begin{align*}
		  W^-:= \mathbb H_2 \cap \{x \in B(o,2R_{n,k}) \setminus
                  ( (B(\tau e_2,\tau)^o
		  \cup (B(-\tau e_2,\tau))^o): x \cdot e_1 < 0 \}.
          \end{align*}
	  For each $i \in I$, the set
	  $B(X_i,R_{n,k})
	\cap \mathbb H_2$ 
	is a disk of radius  denoted $S_i$, 
 where $S_i \leq R_{n,k} \leq \tau(A)/9$, with centre denoted $Y_i$
 (where $Y_i$ is the projection of $X_i$ onto $\mathbb H_2$).
	Therefore since $q' \in W^- \cap \partial B(X_i, R_{n,k})$
	and $B_2(S_ie_2, S_i)$
 (a 2-dimensional ball in $\mathbb H_2$) is contained in
	$B(\tau e_2,\tau)^o \cup \{o\}$ (and hence disjoint from $W^-$),
 the disk $B_2(Y_i, S_i)$ must be obtained
 by a small counterclockwise rotation of $B_2(S_i e_2, S_i)$ about
 $(0,0)$; see Figure \ref{f:tangents} (right).
	Thus, defining the quadrant $\mathbf Q_2 := \{a e_1 + b e_2: a,b
	\in (0,\infty)\}$, we have
$$
B_2( Y_i,S_i) \cap \mathbf Q_2 
	\subset B_2(S_ie_2, S_i)
	\subset B_2(S_Ie_2, S_I),
$$
where we set $S_I := \max_{i \in I} S_i$.
Choose $\delta \in (0,\tau(A)/9)$ with $
\cup_{i \in [n ] \setminus I \setminus I_0} B(X_i,R_{n,k}) 
\cap
B(o,2 \delta) 
= \emptyset$, and let
$$
y \in \mathbf Q_2 \cap B(0,\delta) \cap B(\tau e_2, \tau) \setminus B 
(S_I e_2, S_I);
$$
again see Figure \ref{f:tangents} (right).  Then
$y \in A \setminus
\cup_{i \in [n] \setminus I_0 } B(X_i,R_{n,k})$, contradicting the definition of
$R_{n,k}$.
Therefore $(H_{n,I} \cap \{R_{n,k} \leq \tau(A)/9 \})
\subset (D_n \cup D'_n)^c$, so that
$\Pr[H_{n,I} \cap \{R_{n,k} \leq \tau(A)/9 \}] =0$ as claimed.
\end{proof}

Given $J \subset [n]$ with $|J|=d$, let $q_J$ denote the
location in $P_J$ closest to $\cX_J$, where 
	the set $P_J$ was defined at \eqref{e:defPI}.
	Define the sets
	$$
	P_J^+ := \{x \in P_J: (x-q_J) \cdot e_{\ell(J)} \geq 0\};
	~~~ 
	P_J^- := \{x \in P_J: (x-q_J) \cdot e_{\ell(J)} \leq 0\},
	$$
	where $\ell(J)$ is the first index $i \in [d]$ such
	that $e_i$ is not orthogonal to $P_J$.
	If event $D_n$ occurs,
	where $D_n$ was defined just after 
	\eqref{e:defPI}, then $P_J$ 
	is 1-dimensional, i.e. a line, and 
	$P_J^+$ and $P_J^-$ are half-lines each with endpoint  $q_J$. 
	Next, given $\alpha >0$ define
	\begin{align*}
	R_{J,\alpha}^+ :=
	 \sup \{ \dist(y,\X_J): y \in P_J^+ \cap
	 \partial A, \dist(y, \cX_J) \leq \alpha \}; \\
	R_{J,\alpha}^- :=
	 \sup \{ \dist(y,\X_J): y \in P_J^- \cap
	 \partial A, \dist(y, \cX_J) \leq \alpha \}, 
	\end{align*}
	with the convention that $\sup(\emptyset) := +\infty$.
	Also set $\mathbb Q_+:= \mathbb Q \cap (0,\infty)$.
	The next lemma says that if $R_{n,k} \leq \tau(A)/9$,
	then $R_{n,k} = R^+_{J,\alpha}$ or
	 $R_{n,k} = R^-_{J,\alpha}$ for some 
	 $\alpha \in \mathbb Q_+$ and $J \subset [n]$ with $|J| =d$.
	 This is useful because
	 $R_{J,\alpha}^+$ and $R_{J,\alpha}^-$ 
	 are determined just by $\X_J$ without reference to
	 $\X_{[n] \setminus J}$.
	\begin{lemm}
		\label{l:Rdot}
		Let $n, k \in \N$ with $n \geq k$.
		Suppose $\partial A \in C^{1,1}$ and $B=A$. Suppose 
		$R_{n,k} \leq \tau(A)/9$ and 
	moreover	$D_n \cap D'_n$ occurs.
		Let
		$w \in \cW_{n,k} \cap \partial A$.
		Then there exists $J \subset [n]$ with $|J|= d$,
		$\alpha \in  \mathbb Q_+$, and
		$\bullet \in \{+,-\}$ 
		such that $w \in P_J^\bullet$ and 
		$$
		R_{n,k} = R_{J,\alpha}^\bullet = \dist(w,\cX_J).
		$$
	\end{lemm}
	\begin{proof}
	By Lemma \ref{l:I>=d} and the assumption that $D'_n$ occurs,
	there is a unique set $J \subset [n]$ with $|J|=d$
	such that $w \in P_J$ with $\dist(w,\cX_J) = R_{n,k}$.
	Choose $\bullet \in \{+,-\}$ such that $w \in P_J^\bullet$.
	
	Define the index set
	$
	I_0 := \{i \in [n]: \| X_i - w \| < R_{n,k}\},
	$
		and set $M:= \dist(w, \X_{[n] \setminus J \setminus I_0})$,
		so that $M > R_{n,k}$. 
	Choose rational  $\alpha > R_{n,k}$ such that
	$$
	\alpha - R_{n,k} < \min \Big( R_{n,k}, \frac{(M-R_{n,k})^2}{3 R_{n,k}}
	\Big).
	$$
		To complete the proof, it suffices to
		show that  $R_{n,k} = R_{J,\alpha}^\bullet$.
		Note first that
	$
	\alpha^2 - R_{n,k}^2 = (\alpha+R_{n,k})(\alpha - R_{n,k})
	\leq 3 R_{n,k}(\alpha- R_{n,k}).
	$
	Next, note that
	if $z \in P_J^\bullet \cap \partial A$ with
	$\dist(z,\X_J) \in (R_{n,k},\alpha)$ then
	by the cosine rule
		$\alpha^2 > \dist(z,\X_J)^2 \geq 
		\|z-w\|^2+ R_{n,k}^2$ so
	$ \|z-w\| \leq \sqrt{\alpha^2-R_{n,k}^2} $
		and hence by the triangle inequality
	\begin{align*}
		\dist(z,\X_{[n] \setminus J \setminus I_0} )
		\geq M - 
	\sqrt{\alpha^2-R_{n,k}^2}
		\geq 
		M -  \sqrt{3 R_{n,k} (\alpha - R_{n,k}) } 
		> R_{n,k},
	\end{align*}
	and hence $k$-$\dist(z,\X_n) 
		> R_{n,k}$,
	contradicting the definition of $R_{n,k}$. 
		Hence there is no such $z$, so
		$R_{n,k} = R^\bullet_{J, \alpha}$ as required.
	\end{proof}

\begin{proof}[Proof of Proposition \ref{p:uniquefurth}]
	Let $k, n \in \N$ with $k \leq n$.
	By Proposition \ref{lemWn}, it suffices to prove that
	$\Pr[ \{ |\cW_{n,k}| > 1 \} 
	\cap \{ |\cW_{n,k} \cap A^o | \leq 1 \} 
	\cap \{R_{n,k} \leq \delta_5\}] = 0$, where
	$\delta_5$ is as in Lemma
	\ref{l:fromHPY}.

	  Suppose $\mathcal W_{n,k} \cap \partial A$ and 
	   $ \cW_{n,k} \cap A^o  $ are both non-empty,
	   and
	   $R_{n,k} \leq \tau(A)/9$, and $D_n \cap D'_n$ occurs.
	By Lemma \ref{l:Rdot} there exist $I \subset [n]$
	with $|I|=d$, $\alpha \in  \mathbb Q_+$
	and $\bullet \in \{+,-\}$ with
	$R_{n,k} =R^\bullet_{I,\alpha}$.
	Also  by Lemma \ref{l:I>d} there exists $J \subset [n]$
	with $|J|=d+1$ and $R_{n,k} = \dist(Y_J,\X_J)$. Therefore
	\begin{align*}
		\mathbb P[\{W_{n,k} & \cap \partial A \neq \emptyset\}
	\cap \{W_{n,k} \cap A^o \neq \emptyset \}
		\cap \{R_{n,k} \leq \delta_5\} ]
		\\
		& \leq \sum_{I,J\subset [n]: |I|= d,|J| =d +1}
		\Big(
		\sum_{\alpha \in \mathbb Q_+ , \bullet \in \{+,-\}}
		\Pr [
		\dist(Y_J,\X_J) 
			= 
			R_{I,\alpha}^\bullet 
		] \Big).
	\end{align*}
	Let $I,J \subset [n]$
	with $|I|= d$ and $|J|= d+1$. 
	Select an index
	$j \in J \setminus I$. 
	If $D_n$ occurs, the set $P_{J \setminus \{j\}}$
	is a line and there are at most two elements 
	of the set
	$$
	V := \{y \in P_{J \setminus \{j\}}: \dist (y,\X_{J \setminus \{j\}}) 
	= R_{I,\alpha}^\bullet \},
	$$
	and to have 
		$
		\dist(Y_J,\X_J) 
		= 
		R_{I,\alpha}^\bullet 
		$
		we require that $\|Y_j-u\| = R_{I,\alpha}^\bullet$
		for some $u \in V$,
		 an event of probability zero.
		 Hence
		 \begin{align}
	\mathbb P[\{W_{n,k} \cap \partial A \neq \emptyset\}
	\cap \{W_{n,k} \cap A^o \neq \emptyset \}
		\cap \{R_{n,k} \leq \delta_1\} ]
			 =0.
			 \label{e:0726a}
		 \end{align}

Now suppose $\{|\cW_{n,k} \cap \partial A| > 1\} \cap
	\{R_{n,k} \leq \delta_5 \} \cap D_n \cap D'_n$
	occurs.  Let $y,y'$ be distinct elements of
$\cW_{n,k} \cap \partial A $. 
	By Lemma \ref{l:Rdot} there exist $I,I' \subset [n]$
	with $|I|= |I'|=d$, $\alpha, \alpha' \in  \mathbb Q_+$
	and $\bullet, \bullet' \in \{+,-\}$ with
	$R_{n,k} =R^\bullet_{I,\alpha} = R^{\bullet'}_{I',\alpha'}$ 
	and $y \in P^\bullet_{I}$, $y' \in P^{\bullet'}_{I'}$,
	and moreover $\dist(y,\cX_I) = R_{n,k}
	= \dist(y',\cX_{I'})$.

	If $I = I'$ then event $H_{n,I}$, defined in Lemma \ref{l:Hn},
	occurs. Therefore
	\begin{align*}
		\mathbb P[ \{| \cW_{n,k} \cap \partial A | > 1\}
		& \cap \{R_{n,k} \leq \delta_5 \}]
		 \leq \sum_{I \subset [n]:|I|=d} \mathbb P[H_{n,I}
		 \cap \{R_{n,k} \leq \delta_5\}]
		\\
		& + \sum_{I,I' \subset [n]: |I|= |I'|= d,I \neq I'}
		\Big(
		\sum_{\bullet,\bullet' \in \{+,-\}, \alpha, \alpha'
		\in \mathbb Q_+ 
		}
		\mathbb P
		[ R^\bullet_{I,\alpha}
		= R^{\bullet'}_{I',\alpha'} \leq \delta_5 ]
		\Big).
	\end{align*}
	Let $I, I' \subset [n]$ with $|I|=|I'|=d$ and $I \neq I'$.
	Let $\alpha, \alpha' \in \mathbb Q_+$ and $\bullet, \bullet'
	\in \{+,-\}$.
	Let $j \in I' \setminus I$. Then 
	\begin{align*}
		\Pr[ R^\bullet_{I,\alpha} = R^{\bullet'}_{I',\alpha'} \leq
		\delta_5]
		\leq \mathbb P[ \dist(X_{j},\partial A)=  
		R_{I,\alpha}^\bullet \leq \delta_5],
	\end{align*}
and this probability is zero by Lemma
	\ref{l:fromHPY}.
	Also by Lemma \ref{l:Hn}, 
	$\Pr[H_{n,I} \cap \{R_{n,k} \leq \delta_5\}] =0$ for each $I \subset [n]$ with $|I|=d$.
	Hence 
		$\mathbb P[ \{|\cW_{n,k} \cap \partial A | > 1\}
		 \cap \{R_{n,k} \leq \delta_5 \}] =0$,
		and combined with \eqref{e:0726a} this
		completes the proof.
\end{proof}

\section{Proof of Theorems \ref{t:record} and \ref{Thm2}}
\allco

Recall that 
	$\cW_{n,k}$ and $W_{n,k}$  were defined at \eqref{e:cWdef},
	and that $f_0:= \inf_{x \in B} f(x)$.
\begin{lemm}
\label{lemWn2}
	Suppose $B \subset A^o$.
	Then $f(W_{n,k}) \to f_0$ as $n \to \infty$, almost surely.
\end{lemm}
\begin{proof}
	Let $\eps >0$ and
let
	$
	B_{\eps} := B \cap f^{-1}([f_0 + \eps, \infty)).
	$
If $B_\eps \neq \emptyset$, then by 
	\eqref{0315a} from
	Lemma \ref{lem1}, almost surely
$$
	\lim_{n \to \infty} n \theta R_{n,k}(B_\eps)^d
	= \frac{1}{\inf_{x \in B_\eps} f(x)} < \frac{1}{f_0} = 
	\lim_{n \to \infty} n \theta R_{n,k}(B)^d,
	$$
	so that for large enough $n$ we have
	$
	R_{n,k}(B_\eps) < R_{n,k}(B) =: R_{n,k}$ and hence
	$ \sup_{x \in B_\eps}(k$-$\dist(x,\X_n)) < R_{n,k}$
	so $B_\eps \cap  \cW_{n,k} = \emptyset$ and $W_{n,k}
	\notin B_\eps$ (this also holds if $B_\eps = \emptyset$).
	The result follows.
\end{proof}
\begin{lemm}
\label{lemEn}
	Suppose $B \subset A^o$ and $\dim(\partial B) < d$.
	Let $k \in \N$.
	 Then with probability 1,
	 ${\cal W}_{n,k} \cap \partial B = \emptyset$ 
	 for all but finitely many $n$.
\end{lemm}
\begin{proof}
	We apply Lemma \ref{l:Sn},
	taking $A_{r} := \partial B$
	for all $r >0$.  Choose $\delta \in (0,1)$ such that
	$\dim(\partial B) < d- \delta$, and
	$\eps >0$ with $\eps < \delta f_0/d$. Then $\kappa (\partial B)
	= O(r^{\delta -d})$, and $\mu(B(x,s)) \geq (f_0 - \eps)
	\theta s^d$
	for all $x \in \partial B$ and all
	small enough $s >0$. Therefore by taking
	$a = (f_0 - \eps) \theta$ and $b= d-  \delta$
	in Lemma \ref{l:Sn} we have almost surely that
	$$
	\limsup_{n \to \infty}
	\Big( \frac{n \theta  R_{n,k}(\partial B)^d }{\log n} \Big)  \leq 
	\frac{ 1- \delta/d}{f_0 - \eps} 
	= \frac{1}{f_0}\Big(\frac{1- \delta/d}{1 - \eps/f_0} \Big)
	< \frac{1}{f_0}
	.
	$$
	Comparing with \eqref{0315a} from Lemma \ref{lem1}, we may deduce that
	$R_{n,k}(\partial B) < R_{n,k}(B) =: R_{n,k}$ for all large enough $n$,
	almost surely. But if $R_{n,k}(\partial B) < R_{n,k}$,
	then $\cW_{n,k} \cap \partial B = \emptyset$,
	and the result follows.
\end{proof}

\begin{proof}[Proof of Theorem \ref{t:record}]
	Fix $k \in \N$.
	Set $I_k := 1$, and for each $n \geq k+1$,
	define the random variable 
	$I_n:= {\bf 1}_{\{R_{n,k} < R_{n-1,k}\}}$
	(so that $N_{n,k} = \sum_{i = k}^n I_i$).
	For each $n \geq k$,
	set $U_n := {\bf 1}_{ \{|\cW_{n,k}|=1\}}$, and
	$Y_n:= \mu(B(W_{n,k},R_{n,k}))$, where $W_{n,k} $ was defined
	just after \eqref{e:cWdef}.
	For each $n \in \N$, let $\F_n := \sigma(X_1,\ldots,X_n)$,
	  the $\sigma$-algebra generated by $(X_1,\ldots,X_n)$.

  Then, for each $n \geq k$, the random variables $I_n $, $U_n $ and
	$Y_n$ are  $\F_n$-measurable.
Also, for each $n \geq k+1$, we claim that almost surely
	\begin{align}
		\E [ I_n U_{n-1} | \F_{n-1} ] = Y_{n-1} U_{n-1}. 
		\label{0116c2}
	\end{align} 
	This is because if $U_{n-1} =1$, the
maximum radius open ball centred in $B$ and  
containing fewer than $k$
	points of $\X_{n-1}$ is unique, 
 and in this case $I_n = 1$ if and
only if the next point $X_n$ lands in this ball. Moreover,
	$\E[ I_n | \F_{n-1} ] \leq Y_{n-1} $ a.s.
	because
	$I_n \leq {\bf 1}_{\{X_n \in B(W_{n-1,k},R_{n-1,k}) \}}$.

	By Proposition \ref{lemWn}
	and Lemma \ref{lemEn}, we have that  $\sum_{n=k}^\infty (1-U_n)
	< \infty$, almost surely.
	By \eqref{0315a} from Lemma \ref{lem1}, Lemma \ref{lemWn2}, and the assumed continuity of
$f$ we have that, almost surely, $Y_n \sim (\log n/n)$ as $n \to \infty$.
	Also by Lemma \ref{meanlem} the
	random variables $(nY_n/\log n, n \geq k)$
are uniformly integrable.
	
	Thus all the conditions for
	Lemma \ref{l:genrec} apply here, and
	we  obtain the result using that lemma with $c=1$.
\end{proof}

\begin{lemm}
	\label{l:tWinint}
	Suppose $\partial A \in C^{1,1}$ or $A$ is convex.
	Let $k \in \N$.
	Then
	$\tcW_{n,k } \subset (A^{(R_{n,k})})^o  $ 
	for all but finitely many $n$,
	almost surely.
\end{lemm}
\begin{proof}
	We shall apply Lemma
	\ref{l:Sn}, taking $A_r := A^{(r)} \setminus A^{(3r)}$.
	By Lemma \ref{l:Covbdy}
	we have $\kappa(A^{(r)} \setminus A^{(3r)},r ) = O(r^{1-d}) $
	as $ r \downarrow 0$,
	and for all $r >0$, all $x \in 
	A^{(r)} \setminus A^{(3r)}$ and all $s \in (0,r)$, we have
	$\mu(B(x,s)) \geq f_0 
	\theta s^d$.

	For this proof set
	$S_{n,k}:= \inf \{r>0: A^{(r)} \setminus A^{(3r)} \subset
	F_{n,k,r}\}$. 
	 Applying Lemma \ref{l:Sn} with 
	 $a= f_0 \theta$ 
	 and 
	 $b = d-1$, 
	 and using also \eqref{0315b},
	 we obtain that almost surely
	 $$
	 \limsup_{n \to \infty} (n S_{n,k}^d/\log n)
	 \leq (1-1/d)/(f_0 \theta)
	 < \liminf_{n \to \infty} ( n \tR_{n,k}^d/\log n),
	 $$
	 so that for large enough $n$ we have $S_{n,k} < \tR_{n,k}$.

	 Suppose $S_{n,k} < \tR_{n,k}$. Write $R$
	 for $\tR_{n,k}$. Then there exists $s \in (0, R)$ such that
	 $A^{(s)} \setminus A^{(3s)} \subset F_{n,k,s}$.
	 If $R/2 < s <R$ then
	 $\partial(A^{(R)}) \subset A^{(s)} \setminus A^{(3s)}
	 \subset F_{n,k,s}$,
	 and hence $\tcW_{n,k} \cap \partial (A^{(R)}) = \emptyset$.

	 If instead $s \leq R/2$ then for each $x \in \partial(A^{(R)})$
	 we can find $y \in \partial(A^{(2s)})$
	 with $\|y-x\| = R-2s$
	 (namely, let $z $ be the closest point of $\partial A$ to $x$,
	 and let $y \in [x,z]$ with $\|y-x\| = R-2s$).
	 Then $y \in A^{(s)} \setminus A^{(3s)} \subset F_{n,k,s}$, so
	 by the triangle inequality
	 $x \in F_{n,k,s+(R-2s)} = F_{n,k,R-s}$, so that $x \notin 
	 \tcW_{n,k}$ and again
	 $\tcW_{n,k} \cap \partial (A^{(R)}) = \emptyset$.
\end{proof}
\begin{lemm}
	\label{l:fconv}
	It is the case that $f(\tW_{n,k}) \to f_0 := f_0(A)$
	as $n \to \infty$.
\end{lemm}
\begin{proof}
	Let $\eps >0$ and set $D_\eps:= 
	f^{-1}([f_0+  \eps,\infty))$.  If $D_\eps \neq \emptyset$,
	then 
	by 
	Lemma \ref{l:upper},
	$$
	\limsup_{n \to \infty} n \theta
	\tR_{n,k}(D_\eps)^d/\log n \leq 1/f_0(D_\eps) \leq 1/(f_0+ \eps),
	$$
	and hence comparing with \eqref{0315b},  for large enough $n$, 
	$\tR_{n,k}(D_\eps) < \tR_{n,k}$ so that there exists
	$s < \tR_{n,k}$ with
	$$
	A^{(\tR_{n,k})} \cap D_\eps \subset A^{(s)} \cap D_\eps
	\subset F_{n,k,s},
	$$
	and hence $\tcW_{n,k} \cap D_\eps = \emptyset$, so that 
	$\tW_{n,k} \notin D_\eps$. 
	The result follows.
\end{proof}

\begin{proof}[Proof of Theorem \ref{Thm2}]
	Fix $k \in \N$.
For $n \geq k+1$, define the random variables $I_n := {\bf 1}_{\{\tR_{n,k}
< \tR_{n-1,k}\}}$ (with $I_k =1$ so that $\tN_n= \sum_{j = k}^n I_j$) and
for $n \geq k$ set
	$U_n := {\bf 1}_{\{|\tcW_{n,k}|=1\}}$ and
 $Y_n := \mu(B(\tW_{n,k},\tR_{n,k}))$.

	By Lemma \ref{l:tWinint}
and Proposition \ref{lemWn}, $\sum_{n = k}^\infty (1- U_n) < \infty$ 
almost surely.
	By Lemma \ref{l:fconv}, \eqref{0315b}
	and the continuity of $f$, we have 
	$n Y_n \sim \log n$ as $n \to \infty$.
	Setting $\cF_n := \sigma(X_1,\ldots,X_n)$,
	we also have \eqref{0116c} and \eqref{e:squeeze} for all $n$.
	The random variables $(nY_n/\log n, n \geq k)$
	are uniformly integrable by \eqref{eqmeanlem2} from Lemma \ref{meanlem}.

	Thus Lemma \ref{l:genrec} is applicable (with $c=1$) and
	gives us the result.
\end{proof}

\section{The case $B=A$: proof of Theorem \ref{Thm3}}
\allco
We first consider the case with $B=A$ and with 
\begin{align}
(2 -2/d)/ f_1 > 1/f_0.
	\label{e:bdydom}
\end{align}
For example, this includes the uniform case if $d \geq 3$.

\begin{lemm}
	\label{l:Wtobdy}
	Suppose $B=A$, $\partial A \in C^{1,1}$
	and \eqref{e:bdydom} holds. Then almost surely as $n \to \infty$,
	$$
	(n/(\log n))^{1/d}
	\dist(W_{n,k},\partial A) \to 0.
	$$
\end{lemm}
\begin{proof}
	Let $\eps \in (0,\frac12)$. By Lemma \ref{l:bdyballs}(i)
	there exists $\delta >0$
	such that for all $r \in (0,\delta)$
	and all $x \in A^{(2 \eps r)}$
	we have $\lambda_d(B(x,r)) \geq (1+\eps) \theta r^d/2$.
	If $0 < s < r < \delta$, then $A^{(2 \eps r)} \subset A^{(2 \eps s)}$
	so for $x \in A^{(2 \eps r)}$ we have
	$\lambda_d(B(x,s)) \geq (1+\eps) \theta s^d/2$.

	Choose $f_1^- \in (0,f_1)$ such that $(1+ \eps) f_1^- > f_1$,
	and then (using \eqref{e:bdydom})
	choose $\alpha < (2-2/d)/(f_1 \theta)$ such that
	$\alpha \theta (1+ \eps) f_1^- > 2- 2/d$
	and $\alpha > 1/(f_0 \theta)$. Then set 
	$r_n = (\alpha (\log n)/n)^{1/d}$.

	We shall apply Lemma \ref{lemmeta} taking 
	 $A_r = A^{(2 \eps r)} \setminus A^{(r)}$. 
	For all  sufficiently small $r >0$  we have
	$$
	\mu(B(x,s)) > (1+ \eps) f_1^- \theta s^d/2,  ~~~~ x \in
	 A^{(2 \eps r)} \setminus A^{(r)}, ~~s \in (0,r).  
	$$
	Also by Lemma \ref{l:Covbdy}, 
	$\kappa (A^{(2 \eps r)} \setminus A^{(r)},r)
	\leq \kappa(A \setminus A^{(r)},r) 
	= O(r^{1-d})$
	as $r \downarrow 0$.  By Lemma \ref{lemmeta} with
	$a= (1+ \eps) f_1^- \theta/2$
	and $b=d-1$  so that $b/(ad) = (2-2/d) / ((1+\eps) f_1^- \theta)
	< \alpha$, there exists $\delta' >0$
	such that
	\begin{align}
		\Pr[ \{ A^{(2 \eps r_n)} \setminus A^{(r_n)}
	\subset F_{n,\lfloor \delta' \log n \rfloor, r_n} \}^c]
	= O (n^{-\delta'}) ~~ {\rm as}~~ n \to \infty.
		\label{e:covshell}
	\end{align}

	Also $\kappa(A^{(r)},r) = O(r^{-d})$ as $r \downarrow 0$ and
	$\mu(B(x,r)) \geq f_0 \theta r^d$ for all $x \in A^{(r)}$ so that
	applying  Lemma \ref{lemmeta} to $A^{(r)}$, now
	taking $a = f_0 \theta$ and $b = d$ so
	that $b/(ad) = 1/ (f_0 \theta) < \alpha$, we can find
	 $\delta'' \in(0,\delta')$ such that
	\begin{align}
		\Pr[  \{ A^{(r_n)}
	\subset F_{n,\lfloor \delta'' \log n \rfloor, r_n} \}^c]
	= O (n^{-\delta''}) ~~ {\rm as}~~ n \to \infty.
		\label{e:covint}
	\end{align}
	Let $S_{n,k}:= \inf \{r >0:A^{(2 \eps r)}
	\subset F_{n,k,r} \}$. This is nonincreasing in $n$ and nondecreasing
	in $k$, and by \eqref{e:covshell}, \eqref{e:covint} 
	and the union bound
	\begin{align*}
	\Pr[  n  S_{n,\lfloor \delta'' \log n\rfloor}^d/\log n > \alpha ] 
	= \Pr[ S_{n,\lfloor \delta'' \log n\rfloor}  > r_n] \leq
		\Pr[ \{A^{(2 \eps 
		r_n)} \subset F_{n,\lfloor \delta'' \log n\rfloor,r_n} \}^c]
		\\
	= O(n^{-\delta''}),
		\end{align*}
	so by Lemma \ref{lemtrick}, with probability 1 we have
	$\limsup_{ n \to \infty} n  S_{n,k}^d/\log n \leq \alpha
	$
	so that by \eqref{0315c} and \eqref{e:bdydom}, 
	$S_{n,k} < R_{n,k}$ for all large enough $n$. 
	 
	 But if $S_{n,k} < R_{n,k}$ then there exists $s \in (0, R_{n,k})
	 $ with $A^{(2 \eps R_{n,k})} \subset A^{(2 \eps s)} \subset F_{n,k,s}$,
	 so that
%
$W_{n,k} \notin A^{(2 \eps r_n)}$.
	Since $\eps$ can be arbitrarily small, the result follows.
\end{proof}

\begin{lemm}
	\label{l:ftof1}
	Suppose $B=A$, $\partial A \in C^{1,1}$,
	and \eqref{e:bdydom} holds.
	Then
	$f(W_{n,k}) \to f_1$ as $n \to \infty$, almost surely.
\end{lemm}
\begin{proof}
	Let $\eps > 0$.
	Set $D := f^{-1}([f_1+ \eps, \infty))$,
	and suppose $D \neq \emptyset$.
	Then 
	by Lemma \ref{l:upper},
	$$
	\limsup_{n \to \infty} \Big(\frac{n \theta R_{n,k}(D)^d}{\log n} \Big)
	\leq \max \Big( \frac{1}{f_1 +  \eps},
	\frac{2-2/d}{f_1 +  \eps} \Big)
	= \frac{2-2/d}{f_1+  \eps},
	$$
	where the second inequality is because the condition
	\eqref{e:bdydom} implies $d \geq 2$.

	Therefore using \eqref{0315c} and \eqref{e:bdydom}, 
	almost surely for $n$ large enough we have
	$R_{n,k}(D) < R_{n,k}$, and hence
	$W_{n,k} \notin D$ (this conclusion also holds
	if $D = \emptyset$). Therefore almost surely
	$\limsup_{n \to \infty} f(W_{n,k}) \leq f_1$,
	but also $\liminf_{n \to \infty} f(W_{n,k}) \geq f_1$ 
	by Lemma \ref{l:Wtobdy} and the continuity of $f$,
	and hence the result.
\end{proof}
We now turn to the case where $B=A$ and
\begin{align}
(2 -2/d)/ f_1 < 1/f_0.
	\label{e:intdom}
\end{align}
For example, the  case with $d=1$ is included in this case.
\begin{lemm}
	\label{l:Wtoint}
	Suppose $B=A$, $\partial A \in C^{1,1}$
	and \eqref{e:intdom} holds. Then 
	almost surely,
	\begin{align}
		W_{n,k} \in A^{(R_{n,k})}
	~~~~~~\mbox{\rm for all but finitely many } n.
		\label{e:Wint}
	\end{align}
\end{lemm}
\begin{proof} Let $f_1^- < f_1$ be such that
	$(2-2/d)/f_1^- < 1/f_0$. 

	We shall apply Lemma \ref{l:Sn} taking 
	 $A_r = A \setminus A^{(2r)} $. 
	 By Lemma \ref{l:bdyballs}(iii), for $r$ sufficiently small we have
	\begin{align*}
	\mu(B(x,s)) >  f_1^- \theta s^d/2,  ~~~~ x \in
		A \setminus A^{(2r)}, ~ s \in (0,r].  
		\end{align*}
	Also by Lemma \ref{l:Covbdy} we have
	$ \kappa ( A \setminus A^{(2r)},r  ) = O(r^{1-d}) $
		 as $r \downarrow 0$.
	Thus by Lemma \ref{l:Sn} 
	with $a=  f_1^- \theta/2$ and $b=d-1$, setting
	 $S_{n,k}:= \inf \{r >0: A \setminus A^{(2r)}
	\subset F_{n,k,r} \}$,  with probability 1 we have
	$\limsup_{ n \to \infty} nS_{n,k}^d/\log n \leq 
	(2-2/d) / ( f_1^- \theta) < 1/(f_0 \theta)$, so that 
	using  \eqref{0315c} and \eqref{e:intdom}, we have almost surely that
	for all large enough $n$, $S_{n,k} 
	< R_{n,k}$.

	Suppose $S_{n,k} < R_{n,k}$.
	Then we have for some $s  \in (0, R_{n,k})$
	that $A \setminus A^{(2s)}  \subset F_{n,k,s}$.
	We claim next that $
	\cW_{n,k} \subset 
	A^{(R_{n,k})}.$ 
	To see this, let $x \in A \setminus A^{(R_{n,k})}$.
	If $x \in A \setminus A^{(s)}$ then
	$x \in F_{n,k,s} $ so $k$-$\dist(x,\X_n) \leq s <R_{n,k}$.
	If instead $s \leq \dist(x,\partial A) < R_{n,k}$, let $u$ be the
	closest point of $\partial A$ to $x$, and let
	$y = u + \frac{s}{\|x -u\|} (x-u)$. Then
	$y \in A \setminus A^{(2s)}
	\subset F_{n,k,s}$, but then using the triangle inequality
	we have $x \in F_{n,k,\|x-u\|}$ 
	so that $k$-$\dist(x, \X_n) \leq \|x-u\| < R_{n,k}$.
	Thus in both cases $k$-$\dist(x,\X_n ) < R_{n,k}$ so
	$x \notin \cW_{n,k}$, 
	and the claim follows. This claim gives us \eqref{e:Wint}.
\end{proof}

\begin{lemm}
	\label{l:ftof0}
	Suppose $B=A$, $\partial A \in C^{1,1}$
	and \eqref{e:intdom} holds. Then
	almost surely $f(W_{n,k} ) \to f_0$ as $n \to \infty$.
\end{lemm}
\begin{proof}
	Using
	\eqref{e:intdom},
	let $\eps >0$ be  such that  $1/(f_0+ \eps) > (2-2/d)/f_1$.
	Define the set $D := f^{-1}([f_0 + \eps ,\infty))$. 
	If $D \neq \emptyset$, then
	by Lemma \ref{l:upper}, 
	\begin{align*}
	\limsup_{n \to \infty}
		n \theta R_{n,k}(D)^d/\log n & 
		\leq \max \Big( \frac{1}{f_0(D)},\frac{
		2-2/d}{f_1(D)} \Big)
		\\ 
		&\leq \max \Big( \frac{1}{f_0+ \eps},\frac{2-2/d}{f_1} \Big)
		= \frac{1}{f_0 + \eps},
	\end{align*}
	and comparing with \eqref{0315c},
	we obtain for all large enough $n$ that $R_{n,k}(D) < R_{n,k}$
	and  $W_{n,k} \notin D$. The result follows.
\end{proof}

\begin{proof}[Proof of Theorem \ref{Thm3}]
	If \eqref{e:bdydom} holds,
	then as $n \to \infty$,
	$n \theta R_{n,k}^d/\log n \to (2-2/d)/f_1$
	by \eqref{0315c}
	and $\dist(W_{n,k}, \partial A)/R_{n,k} \to 0$
	by Lemma \ref{l:Wtobdy}.
	If instead \eqref{e:intdom} holds, then 
	$n \theta R_{n,k}^d/\log n \to 1/f_0$
	by \eqref{0315c}.

	Then using Lemma \ref{l:ftof1} and Lemma
	\ref{l:bdyballs}(ii) in the first case, and using Lemmas \ref{l:Wtoint}
	and \ref{l:ftof0} in the second case,
	we obtain that almost surely
	$$
	\lim_{n \to \infty} \Big( \frac{n \mu(B(W_{n,k},R_{n,k}))}{
		\log n}
	\Big) = \begin{cases}
		(1-1/d) & \mbox{ if } \eqref{e:bdydom} \mbox{ holds} 
		\\
		1  & \mbox{ if } \eqref{e:intdom} \mbox{ holds.}
	\end{cases}
	$$
	
	Now set $I_n := {\bf 1}_{\{R_{n,k} < R_{n-1,k}\}}$, 
	$U_n := {\bf 1}_{\{|\cW_{n,k}|=1\}}$, $Y_n:= \mu(B(W_{n,k},
	R_{n,k}))$, 
	and $\cF_n := \sigma(X_1,\ldots,X_n)$. 
	Then by Proposition \ref{p:uniquefurth}
	and \eqref{0315c}, $\sum_{n=k}^\infty(1-U_n) < \infty $ almost surely.
	Also \eqref{0116c} and  \eqref{e:squeeze} hold as in
	the proof of Theorem \ref{t:record},
	so Lemma \ref{l:genrec} is applicable with $c= 1-1/d$
	if \eqref{e:bdydom} holds and with $c=1$ if \eqref{e:intdom} holds
	(the uniform integrability condition comes from Lemma \ref{meanlem}).
	This gives us \eqref{0720a}, \eqref{0720b} and \eqref{0720c}
	if \eqref{e:bdydom} holds and \eqref{0125a}, \eqref{0125b} and \eqref{0125c}
	if \eqref{e:intdom} holds.
\end{proof}

\section{Related quantities}
\label{s:related}
\allco

To illustrate the wider applicability of our methods, we describe
two further quantities for which they are relevant, without
presenting detailed proofs.

First, consider records for the $k$-coverage threshold $R_{n,k}$
when $B=A$ is a {\em convex  finite polytope} in $\R^d$ (finiteness
means it has finitely many faces). 
Given a convex finite
polytope $A$ in $\R^d$, let $\Phi(A)$ be the class of all faces
of $A$ (counting $A$ itself as a $d$-dimensional face).
Given a face $\varphi \in \Phi(A)$, let 
$D(\varphi)$ denote its dimension (so $0 \leq D(\varphi) \leq d$).
Let $\varphi^o$ denote the relative interior of $\varphi$.
Let ${\cal K} $ be the cone such that every $x \in \varphi^o$
has a neighbourhood $U_x$ such that $A \cap U_x =
(x + {\cal K}_\varphi) \cap U_x$.
Let $\rho_\varphi := \lambda_d({\cal K}_\varphi \cap B(o,1))$,
the so-called {\em angular volume} of face $\varphi$.

Given also a continuous probability density function
$f$ on $A$, let
$f_\varphi := \inf_{x \in \varphi} f(x)$.
Given $j \in \{0,1,\ldots,d\}$ let $\Phi_j$ be the collection
of all faces of $A$ of dimension $j$.
Suppose the maximum
$$
\max_{j \in \{1,\ldots,d\}} \Big( \frac{ j/d}{\max_{\varphi \in \Phi_j}
f_\varphi \rho_\varphi} \Big)
$$
is achieved for a single value of $j$, denoted $j_0$ say. 
By \cite[Theorem 4.2]{MPCov} this maximum is the almost sure
large-$n$ limit
of $n \theta R_{n,k}^d/\log n$. 

Then it
should be possible to prove using the methods of this paper that
in this situation
$N_{n,k} \sim (j_0/(2d)) (\log n)^2$ a.s., and
$\mathbb E[N_{n,k}] \sim (j_0/(2d)) (\log n)^2$, and
moreover
$\nu_{n,k}^{1/\sqrt{n}} \to \exp((2d/j_0)^{1/2})$ 
almost surely
as $n \to \infty$.

Next we consider a different quantity in stochastic geometry,
namely the {\em largest $k$-nearest neighbour link}, which we
denote $L_{k,n}$ and define by
$$
L_{n,k} := \max_{x \in \X_n} k{\mbox -}\dist(x,\X_n \setminus \{x\}).
$$
As before, 
this can be considered as the size of the largest
`hole' in the sample $\X_n$, provided we now 
define a `hole' to be an open ball {\em centred on one of the points of
the sample} $\X_n$ that contains fewer than $k$ other points
of the sample. The largest $k$-nearest neighbour link is of interest
in its own right in stochastic geometry, and also as an asymptotically
sharp bound for the {\em $k$-connectivity threshold}.
See \cite{Poly,PYconnect} for further discussion.

It is not hard to see that the point $X_i \in \X_n$ satisfying
$k$-$\dist(X_i, \X_n \setminus \{X_i\})= L_{n,k}$ is almost surely unique,
and moreover that we have $L_{n+1,k} < L_{n,k}$ if and only if
$X_{n+1}$ lands in the open ball of radius $L_{n,k}$ centred on this point. 
Moreover,  $L_{n,k}$ enjoys the same almost sure large-$n$ asymptotics 
as those for  $R_{n,k}$ (for $B=A$) that we gave in
\eqref{0315c} when $\partial A \in C^{1,1}$; see \cite{LNNL}
(the $C^2$ boundary condition imposed there can be relaxed to a 
$C^{1,1}$ boundary condition). Therefore if we define
$$
N^*_{n,k} := 1 + \sum_{i=k+2}^n {\bf 1}_{\{L_{i,k} < L_{i-1,k}\}},
$$
it should be possible to prove the same asymptotics for
$N_{n,k}^*$ and for $\E[N^*_{n,k}]$ as were given for
$N_{n,k}$ and $\E[N_{n,k}]$
in Theorem \ref{Thm3}. Moreover using results from
\cite{Poly}, It should be possible to show the same asymptotics
for 
$N_{n,k}^*$ and for $\E[N^*_{n,k}]$ as we have proposed for
$N_{n,k}$ and $\E[N_{n,k}]$ in the case where $B=A$ is a
convex finite polytope.

However, we would like to emphasise that $N^*_{n,k}$ is {\em not}
the number of lower records of the sequence $(L_{m,k})_{m \geq k+1}$
up to time $n$
in the conventional sense; we refer to it instead as
the number of {\em pseudo-records}.
Indeed, the sequence $(L_{n,k})_{n \geq k+1}$
is {\em not} monotone non-increasing. Considering just the case $k=1$,
note that if a new point
$X_{n+1}$ arrives outside the covered region $\cup_{i=1}^n B(X_i,L_{n,1})$
then we obtain $L_{n+1,1} > L_{n,1}$, and we would expect that this happens
infinitely often.
It is true that the general
trend of the sequence $(L_{n,k})_{n \geq k+1}$ is downwards, and it
may be that the asymptotics of the true number of records of $L_{n,k}$
are the same as for the number of pseudo-records
$N^*_{n,k}$ but it would need some work and
probably new ideas to
prove this.

\end{document}